\documentclass[11pt]{article}
\usepackage{fullpage}
\usepackage{algorithmic}
\usepackage{algorithm}
\usepackage{epsfig}
\usepackage{graphicx}
\usepackage{latexsym}
\usepackage{amsmath}
\usepackage{amsfonts}
\usepackage{multirow}
\usepackage{amssymb}
\usepackage[sort,numbers]{natbib}
 \usepackage{tikz}

\usetikzlibrary{shapes,arrows,shadows}
\usepackage{amsmath,bm,times}

\usepackage{mathrsfs}
\usepackage{pifont}
\usepackage{yhmath}
\usepackage{bbm}
\usepackage{amsthm}
\usepackage{booktabs}
\usepackage{stmaryrd}
\usepackage{wasysym}
\usepackage{bbm}
\usepackage{array}
\usepackage{enumerate}
\usepackage{dsfont}

\usepackage{enumitem}      %

\numberwithin{equation}{section}
\numberwithin{figure}{section}

\newtheorem{theorem}{Theorem}[section]

\newtheorem{corollary}[theorem]{Corollary}

\newtheorem{definition}[theorem]{Definition}
\newtheorem{remark}[theorem]{Remark}

\theoremstyle{remark}

\makeatletter
\newcommand\figcaption{\def\@captype{figure}\caption}
\newcommand\tabcaption{\def\@captype{table}\caption}
\makeatother

\DeclareMathAlphabet{\mathpzc}{OT1}{pzc}{m}{it}

\begin{document}
\newcounter{my}
\newenvironment{mylabel}
{
\begin{list}{(\roman{my})}{
\setlength{\parsep}{-1mm}
\setlength{\labelwidth}{8mm}
\usecounter{my}}
}{\end{list}}

\newcounter{my2}
\newenvironment{mylabel2}
{
\begin{list}{(\alph{my2})}{
\setlength{\parsep}{-0mm} \setlength{\labelwidth}{8mm}
\setlength{\leftmargin}{3mm}
\usecounter{my2}}
}{\end{list}}

\newcounter{my3}
\newenvironment{mylabel3}
{
\begin{list}{(\alph{my3})}{
\setlength{\parsep}{-1mm}
\setlength{\labelwidth}{8mm}
\setlength{\leftmargin}{10mm}
\usecounter{my3}}
}{\end{list}}

\title{\bf Log-Sobolev Inequality for Wolff Dynamics and Application to the Condensation of Eigen Microstate in the 1D Ising Model }
\author{Cui Kaiyuan\thanks{Academy of Mathematics and Systems Science, University of Chinese Academy of Sciences, Chinese Academy of Sciences, Beijing 100190, China,({ cuiky@amss.ac.cn})}\and Gong Fuzhou \thanks{Academy of Mathematics and Systems Science, University of Chinese Academy of Sciences, Chinese Academy of Sciences, Beijing 100190, China,({ fzgong@amt.ac.cn})}}
	\date{}
\maketitle

\vspace{-5em}

\begin{center}\large

\end{center}


\begin{abstract}
The Wolff dynamics is a non-local Markov chain widely used for simulating the Ising model due to its effectiveness in reducing critical slowing down compared to the Glauber dynamics. Despite extensive algorithmic and numerical studies, a rigorous probabilistic understanding remains limited. In this paper, we take a first step toward addressing this gap. For the one-dimensional (1D) Ising model, we first derive the transition probabilities of the Wolff dynamics and show that, at the critical point, it converges to the two fully aligned configurations and subsequently oscillates between them. This behavior is absent in the Glauber dynamics. Second, we establish a log-Sobolev inequality with an explicit constant for the Wolff dynamics in the entire subcritical regime and derive quantitative bounds on its ergodic averages. As a by-product, at infinite temperature, the obtained constant coincides with the classical log-Sobolev constant of the random walk on the hypercube. Finally, we apply these results to analyze the spectrum of the sample covariance matrix generated by the Wolff dynamics, which was used by Chen et al. to study condensation of eigen microstate. We prove that the spectral behavior agrees with their simulations in the 1D Ising model, thereby providing theoretical support for their findings.\\

\end{abstract}
\section{Introduction}\label{sec:1}
The Ising model, introduced by Lenz \cite{Lenz1920} and Ising \cite{Ising1925} as a model for ferromagnetism, is a fundamental model in statistical physics and has had a profound impact on mathematics, physics, and computer science. A central problem is to efficiently simulate the model and sample from the associated Gibbs measure. Classical approaches are based on single-site update Markov chains such as the Glauber dynamics, which are known to suffer from critical slowing down near the critical temperature. To overcome this difficulty, Swendsen and Wang \cite{SWENDSEN1987} proposed a cluster algorithm inspired by ideas from percolation theory. Shortly thereafter, Wolff \cite{Wolff1989} introduced another percolation-based algorithm, now known as the  Wolff dynamics (also called the Wolff algorithm), which reduces critical slowing down even more effectively. These Markov chains have been widely used in algorithmic and numerical studies. From a probabilistic perspective, a key tool in analyzing convergence to equilibrium is provided by functional inequalities such as the Poincaré inequality, the modified log-Sobolev inequality, and the log-Sobolev inequality. While these inequalities have been extensively studied for the Glauber dynamics and the Swendsen-Wang dynamics (see e.g. \cite{Nam2019,Eldan2022document,Blanca2022,Ullrich2013,Ullrich2014}), to the best of our knowledge, corresponding results for the Wolff dynamics remain largely unexplored.

In this paper, we study the Wolff dynamics for the one-dimensional (1D) Ising model and establish a log-Sobolev inequality with an explicit constant in the entire subcritical regime. First, we derive the transition probabilities of the Wolff dynamics for the 1D Ising model with coupling constant $\hat{J}$, which admit an explicit representation based on a decomposition into connected components of a configuration on $\mathbb{Z}$. Moreover, we verify reversibility (i.e., the detailed balance condition) and ergodicity in the entire subcritical regime (i.e., $\hat{J}\in[0,+\infty)$). At the critical point (i.e., $\hat{J}=+\infty$), we show that, for any initial configuration $Y_{1}\in\Omega:=\{-1,+1\}^{N}$, the Wolff dynamics eventually converges almost surely to the set $\{-\textbf{1},+\textbf{1}\}$ and subsequently oscillates between these two configurations, where  $\textbf{1}:=(1,\cdots,1)^{T}$ denotes the all-ones vector in $\mathbb{R}^N$ and $N$ is the number of spins. In contrast, the Glauber dynamics does not exhibit this behavior. Finally, we establish both a log-Sobolev inequality and a Poincar\'e inequality for the Wolff dynamics with explicit constants in the entire subcritical regime, which yield quantitative bounds on its ergodic averages. As a by-product, at infinite temperature (i.e., $\hat{J}=0$), the Wolff dynamics reduces to the simple random walk on the hypercube $\Omega=\{-1,+1\}^{N}$,  and the corresponding log-Sobolev constants recover the classical  constants of Gross \cite{Gross1975document} and Diaconis and Saloff-Coste \cite{Diaconis1996}.

As an application of the above results, we study the spectral properties of the sample covariance matrix constructed from $M$ microstates generated by the Wolff dynamics \cite{Wolff1989}. Here $M$ denotes the number of samples. This problem is motivated by recent works of Chen Xiaosong and collaborators \cite{hu2019ducument,sun2021ducument}. Using Monte Carlo simulations based on the Wolff dynamics, they observed the following behavior for the 1D Ising model: as $M\to+\infty$ and $N\to+\infty$, the largest eigenvalue of the sample covariance matrix vanishes for $\hat{J}\in[0,+\infty)$, whereas at the critical point $\hat{J}=+\infty$ it converges to a finite nonzero limit. This phenomenon has been interpreted as a condensation of eigen microstate, analogous to the Bose-Einstein condensation.

Building on our quantitative estimates for the Wolff dynamics, we provide a rigorous analysis of this phenomenon in the 1D Ising model. In particular, we show that the largest eigenvalue $\lambda_1(\mathbb{K})$ of the sample covariance matrix $\mathbb{K}$ converges almost surely to $1$ as $M\to+\infty$ at $\hat{J}=+\infty$, and to $1/N$ at $\hat{J}=0$. Moreover, for $\hat{J}\in[0,+\infty)$, the spectral radius of $\mathbb{K}$ converges to $0$ both almost surely under the iterated limit $M\to+\infty$, $N\to+\infty$, and in $L^2(\mathbb{P}_{\mu_{N}})$ under the double limit with $N=o(\sqrt{M})$ (i.e., $N/\sqrt{M} \to 0$ as $M\to+\infty$), where $\mathbb{P}_{\mu_{N}}$, to be defined in Section~\ref{dec:logsobolevwolff}, denotes the law of the Wolff dynamics started from the Gibbs measure $\mu_{N}$ of the 1D Ising model. These results are consistent with simulations reported in \cite{hu2019ducument,sun2021ducument} and provide theoretical support for their findings in the 1D Ising model. A more detailed discussion of the consistency with simulations is provided in Section~\ref{sec:interpretation}.

The paper is organized as follows. In Section~\ref{sec:2}, we collect preliminaries used in the proofs. In Section~\ref{sec:3}, for the 1D Ising model, we derive the transition probabilities of the Wolff dynamics and analyze its basic properties. We also establish the log-Sobolev and Poincar\'e inequalities in the entire subcritical regime. In Section~\ref{sec:converspectralradius}, we investigate the convergence of the spectral radius of $\mathbb{K}$ as $M\to+\infty$ and $N\to+\infty$, building on the results of Section~\ref{sec:3}. Specifically, Section~\ref{sec:background} briefly reviews the background and motivation, and Section~\ref{sec:renor} analyzes the spectrum of $\mathbb{K}$ at the two distinguished temperatures corresponding to fixed points of the renormalization group (RG) transformation. Sections~\ref{sec:sec:almost} and~\ref{sec:4} estimate the spectral radius of $\mathbb{K}$ in the entire subcritical regime under both the iterated and double limits. Finally, Section~\ref{sec:interpretation} summarizes these results and discusses their agreement with simulations.

\section{Preliminaries}\label{sec:2}
In this section, we present key concepts that will be used throughout the paper.
\subsection{Notation}
We consider the 1D Ising model with nearest-neighbor interactions under periodic boundary conditions (i.e., $\sigma_{N+1} \equiv \sigma_1$). When there is no external field, the Hamiltonian is given by
\begin{align}\label{eq:1.1}
	\mathcal{H}(\sigma)=-J\sum_{i=1}^{N}\sigma_{i}\sigma_{i+1}, \qquad \sigma=(\sigma_{1},\cdots,\sigma_{N})\in\Omega=\{-1,+1\}^{N},
\end{align}
where $J>0$ is the ferromagnetic coupling constant. Let $\beta=1/T$ denote the inverse temperature, and define $\hat{J}:=\beta J$, which will be referred to as the coupling constant when no confusion arises. The Gibbs measure associated with the Hamiltonian $\mathcal{H}$ is given by
\begin{align}\label{eq:isinggibbs}
	\mu_{N}(\sigma)=\frac{1}{Z_{N}}\exp\left(\hat{J}\sum_{i=1}^{N}\sigma_{i}\sigma_{i+1}\right), 
\end{align}
where $Z_N$ is the normalizing constant, also known as the partition function in statistical mechanics, given by
\begin{align*}
	Z_{N}=\sum_{\sigma\in\Omega}\exp\left(\hat{J}\sum_{i=1}^{N}\sigma_{i}\sigma_{i+1}\right).
\end{align*}
This model is exactly solvable, and $Z_{N}$ can be written as $$Z_{N}=\lambda_{+}^{N}+\lambda_{-}^{N},$$ where $\lambda_{+}=e^{\hat{J}}+e^{-\hat{J}},\lambda_{-}=e^{\hat{J}}-e^{-\hat{J}}$, see \cite{Morandi2004document} for more details. We denote by $E_{\mu_{N}}[\cdot]$ the expectation with respect to the measure $\mu_{N}$, and define the covariance
\begin{align*} 
	\langle f;g\rangle_{\mu_{N}}:=\text{Cov}_{\mu_{N}}(f;g)=E_{\mu_{N}}[fg]-E_{\mu_{N}}[f]E_{\mu_{N}}[g].
\end{align*}
Under the thermodynamic limit, the two point correlation function 
\begin{align*} \mathcal{G}_{i,j}:=\lim\limits_{N\rightarrow+\infty}\langle \sigma_{i};\sigma_{j}\rangle_{\mu_{N}},
\end{align*}
admits the explicit expression
\begin{align}\label{2.10}
	\mathcal{G}_{i,j}=\left (\frac{\lambda_{-}}{\lambda_{+}}\right)^{|i-j|}=e^{-\frac{|i-j|}{\xi}},
\end{align}	
where $\xi$ is the correlation length
\begin{align*}			\xi:=&-\frac{1}{\log (\boldsymbol{\theta})},
\end{align*}	
and
\begin{align*}			
	\boldsymbol{\theta}:=&\frac{\lambda_{-}}{\lambda_{+}}=\frac{e^{\hat{J}}-e^{-\hat{J}}}{e^{\hat{J}}+e^{-\hat{J}}}.
\end{align*}	
We note that $\log$ denotes the natural logarithm throughout the paper. Hence,
\begin{align*}			
	\xi=\left\{\begin{array}{ccc}
		+\infty,& \qquad T=0, \hat{J}=+\infty, \\
		0,& \qquad  T=+\infty, \hat{J}=0.
	\end{array}
	\right.
\end{align*}
\begin{remark}\label{rem:1.1}
	For the 1D Ising model, it is well known that no phase transition occurs. Nevertheless, following \cite{hu2019ducument}, the zero temperature ($\hat{J}=+\infty$) can be regarded as a critical point $T_c$, in the sense that the correlation length diverges.
\end{remark}
\subsection{The sample covariance matrix and the correlation matrix of microstates}\label{sec:2.1}
Given an initial configuration $Y_{1}$ in the configuration space  $\Omega=\{-1,+1\}^{N}$, we generate samples from the Ising model via the Wolff dynamics \cite{Wolff1989}. For $1\leq i\leq M$, the $i$-th microstate is represented by an $N$-dimensional vector
\begin{align*}
	X_{i}:=(X_{1i},X_{2i},\cdots,X_{Ni})^{T}=\frac{1}{\sqrt{N}}Y_{i},
\end{align*} 
where
\begin{align*}
	Y_{i}:=(Y_{1i},Y_{2i},\cdots,Y_{Ni})^{T},\qquad Y_{ki}=\pm1,\ k=1,\cdots,N.
\end{align*} 
We collect $M$ such microstates to form the statistical ensemble $\mathbb{X}=(X_{1},\cdots,X_{M})$, which is an $N\times M$ matrix. We then define the $M\times M$ correlation matrix of microstates by
\begin{align*}
	\mathbb{C}:=\frac{1}{M}\mathbb{X}^{T}\mathbb{X}.
\end{align*}
Its entries are given by
\begin{align*}
	\mathbb{C}_{ij}=\frac{1}{M}X^{T}_{i} X_{j}=\frac{1}{M}\sum_{k=1}^{N}X_{ki} X_{kj},\qquad i,j=1,\cdots M,
\end{align*}
where $\mathbb{C}_{ij}$ characterizes the correlation between microstates $i$ and $j$. In particular, $X^{T}_{i} X_{i}=1$ and $\text{Tr}(\mathbb{C})=1$, where $\text{Tr}(\mathbb{C})$ denotes the trace of matrix $\mathbb{C}$. 
Moreover, we define the $N\times N$ sample covariance matrix \begin{align}\label{eq:defcov}
	\mathbb{K}:=\frac{1}{M}\mathbb{X}\mathbb{X}^{T},
\end{align}
with elements
\begin{align*}	\mathbb{K}_{ij}=\frac{1}{M}\sum_{r=1}^{M}X_{ir} X_{jr},\qquad i,j=1,\cdots N,
\end{align*}
where $\mathbb{K}_{ij}$ characterizes the correlation  
between the spin $i$ and spin $j$. 
\subsection{Log-Sobolev and Poincar\'e inequalities}\label{sec:Dirichletform}
Let $(\Omega,\mu_{N})$ be a probability space, and let $P$ be a reversible Markov transition kernel with invariant measure $\mu_{N}$. The associated Dirichlet form is defined by
\begin{align*}	
	\mathcal{E}_{\mu_{N}}(f):=\langle f,(I-P)f\rangle_{\mu_{N}}=\frac{1}{2}\sum_{\sigma,\eta\in\Omega}[f(\sigma)-f(\eta)]^{2}P(\eta,\sigma)\mu_{N}(\eta),
\end{align*}
where $\langle\cdot, \cdot\rangle_{\mu_{N}}$ denotes the inner product on $L^{2}(\mu_{N})$ and $f$ is a real-valued function on $\Omega$. 
\begin{definition}\label{def:logsobolev}
	We say that $P$ satisfies a log-Sobolev inequality with constant $C_{\text{LS}}$ if
	\begin{align*} 
		\mathrm{Ent}_{\mu_{N}}(f)
		:= E_{\mu_{N}}[f\log f]-E_{\mu_{N}}[f]\log E_{\mu_{N}}[f]\leq&C_{\text{LS}}\mathcal{E}_{\mu_{N}}(\sqrt{f}),
	\end{align*}
	for all non-negative functions $f$ defined on $\Omega$.
\end{definition}
The Poincar\'e inequality follows from the log-Sobolev inequality via linearization, see \cite{Ledoux2004,Bakry2014} for more details.
\begin{definition}\label{def:poincare} We say that $P$ satisfies a Poincar\'e inequality with constant $C_{\text{PI}}$ if
	\begin{align*} 
		\mathrm{Var}_{\mu_{N}}(f)
		:=E_{\mu_{N}}[f^{2}]-(E_{\mu_{N}}[f])^{2}\leq&C_{\text{PI}}\mathcal{E}_{\mu_{N}}(f),
	\end{align*}
	for all functions $f$ defined on $\Omega$.
\end{definition}
\begin{remark}\label{rem:lsirelatepoin}
	For any chain $P$, the log-Sobolev constant $C_{\text{LS}}$ and the Poincar\'e constant $C_{\text{PI}}$ are related by $2C_{\text{PI}}\leq C_{\text{LS}}$, see Lemma 3.1 in \cite{Diaconis1996} for instance.
\end{remark}
\section{The Wolff dynamics and log-Sobolev inequality}\label{sec:3}
The Wolff dynamics is a non-local Markov chain for simulating the Ising model, introduced by Wolff (1989) as a single-cluster variant of the Swendsen-Wang dynamics. In contrast to local update algorithms such as the Glauber or Metropolis dynamics, which flip a single spin at a time, the Swendsen-Wang and Wolff dynamics update entire clusters of aligned spins. In the Swendsen-Wang dynamics, clusters are constructed via a bond percolation step and flipped independently with probability $\frac{1}{2}$, whereas the Wolff dynamics grows a single cluster from a randomly chosen seed and flips it with probability $1$. As demonstrated in \cite{Luijten2006}, critical slowing down in the Wolff dynamics is reduced even more effectively than in the Swendsen-Wang dynamics.
\subsection{The Wolff dynamics}\label{dec:wolff}
For completeness, we briefly describe the procedure of the Wolff dynamics (see e.g. \cite{Luijten2006,Tamayo1990ducument}) for the Ising model as follows:
\begin{enumerate}
	\item[1.] A single site  $i_{0}$ of the lattice is selected at random to build the cluster.
	\item[2.] All nearest neighbors $j$ of this single site $i_{0}$ are added to the cluster with a probability $\tilde{p}_{i_{0}j}=1-\exp\left(-\hat{J}(1+\sigma_{i_{0}}\sigma_{j})\right)$, provided spins  are parallel and the bond between $i_{0}$ and $j$ has not been considered before.
	\item[3.] Each site $j$ that is indeed added to the cluster is also placed on the stack.
	Once all neighbors of $i$ have been considered for inclusion in the cluster, a site is retrieved from the stack and all its neighbors are considered in turn for inclusion in the cluster as well, following step 2.
	\item[4.] Steps 2 and 3 are repeated iteratively until the stack is empty.
	\item[5.] Once the cluster has been completed, all spins that belong to the cluster
	are flipped.
\end{enumerate}
Having described the practical implementation of the Wolff dynamics, we now proceed to formulate its transition kernel in probabilistic terms. To this end, we introduce the necessary notation for analyzing the key properties of the Wolff dynamics in the 1D Ising model. Let $\Omega=\{-1,+1\}^N$ be the finite configuration space, and let $(Y_k)_{k\in\mathbb{Z}^+}$ be a discrete-time Markov chain on $\Omega$, defined on a probability space $(\mathfrak{E},\mathfrak{F},\mathbb{P})$, which evolves according to the Wolff dynamics described above.

Let $G=([N],\mathcal{G})$ be an undirected graph, where $[N]:=\{1,\cdots,N\}\subset\mathbb{Z}$ is the vertex set and $\mathcal{G}$ is the edge set. Under periodic boundary conditions, an edge $e=(e_{-},e_{+})$ belongs to $\mathcal{G}$ if and only if $e_{+}-e_{-}=1$ or $(e_{-},e_{+})=(N,1)$, and we write $i\sim j$ if $(i,j)\in\mathcal{G}$. For any subset $A\subset[N]$, let $\sigma^{A}$ denote the configuration obtained from $\sigma\in\Omega$ by flipping the spins at all sites in $A$. For a given configuration $\sigma\in\Omega$, define
\begin{align}\label{eq:decomposet}
	\mathcal{B}^{\sigma}_{+}:=&\{e\in\mathcal{G}\mid\sigma_{e_{+}}\sigma_{e_{-}}=+1\}, \nonumber\\
	\mathcal{B}^{\sigma}_{-}:=&\{e\in\mathcal{G}\mid\sigma_{e_{+}}\sigma_{e_{-}}=-1\},\\
	\mathcal{C}^{\sigma}:=&\{k\in[N]\mid\sigma_{k}=+1\}, \nonumber\\
	\mathcal{D}^{\sigma}:=&\{k\in[N]\mid\sigma_{k}=-1\}. \nonumber
\end{align}
For any $\mathcal{C}\subset[N]$, define its edge boundary
\begin{align*}
	\partial \mathcal{C}:=\{(i,j)\in\mathcal{G}\mid i\in\mathcal{C},\, j\notin\mathcal{C}\},
\end{align*}
and for any $\mathcal{V}\subset[N]$, define its vertex boundary
\begin{align*}
	\hat{\partial} \mathcal{V}:=\{k\in \mathcal{V}\mid k+1\notin \mathcal{V}\ \text{or}\ k-1\notin \mathcal{V}\}.
\end{align*}
Obviously, $\#\{\mathcal{C}^{\sigma}\}+\#\{\mathcal{D}^{\sigma}\}=N$, where $\#\{V\}$ is the cardinality of $V\subset\mathbb{Z}$. The sets $\mathcal{C}^{\sigma}$ and $\mathcal{D}^{\sigma}$ form a partition of $[N]$, and each can be decomposed into pairwise disjoint connected components
\begin{align}\label{eq:decom}
	\mathcal{C}^{\sigma}=\bigcup_{r=1}^{\tilde{\mathcal{C}}^{\sigma}}\mathcal{C}_{r}^{\sigma},\qquad 	\mathcal{D}^{\sigma}=\bigcup_{r=1}^{\tilde{\mathcal{D}}^{\sigma}}\mathcal{D}_{r}^{\sigma},
\end{align}
where $\mathcal{C}_{r}^{\sigma}$ and $\mathcal{D}_{r}^{\sigma}$ denote connected components, and $\tilde{\mathcal{C}}^{\sigma}$ and $\tilde{\mathcal{D}}^{\sigma}$ are their respective numbers. 

Here, a subset $A \subset [N]$ is called connected if it is the vertex set of a cluster $\mathfrak{A}$, where a cluster is defined as a connected subgraph of $G=([N],\mathcal{G})$. Under periodic  boundary conditions, one of the following three cases holds:
\[
\tilde{\mathcal{C}}^{\sigma}=\tilde{\mathcal{D}}^{\sigma}, \quad
\tilde{\mathcal{C}}^{\sigma}=1,\ \tilde{\mathcal{D}}^{\sigma}=0, \quad \text{or} \quad
\tilde{\mathcal{C}}^{\sigma}=0,\ \tilde{\mathcal{D}}^{\sigma}=1.
\]
Based on the above decomposition, we derive the transition probabilities of the Wolff dynamics for the 1D Ising model. The resulting expression highlights a key difference from the classical Glauber and Swendsen-Wang dynamics, where the transition kernel assigns positive probability to leaving the configuration unchanged.
\begin{theorem}\label{thm: transitionmatrix}
	Let $(\{Y_{k}\}_{k\in\mathbb{Z}^{+}},\mathbb{P})$ be the Markov chain generated by the Wolff dynamics, and let  $\mu_{N}$ denote the Gibbs measure of the 1D Ising model defined in Equation~\eqref{eq:isinggibbs}. Then the following conclusions hold.
	\begin{enumerate}
		\item For $\hat{J}\in(0,+\infty)$, consider a configuration $\sigma\in\Omega$ and a connected subset $A\subset[N]$ such that $\sigma_i\sigma_j=+1$ for all $i,j\in A$. The transition probability  $P_{\text{W}}(\sigma,\sigma^{A})$ is then given by
		\begin{equation}\label{eq:wolffedgecluster}\begin{split}
				P_{\text{W}}(\sigma,\sigma^{A})
				:=&	\mathbb{P}(Y_{2}=\sigma^{A}\mid Y_{1}=\sigma)\\
				=&\frac{\#\{A\}}{N}\begin{cases}
					\kappa^{\#\{A\}-1}\hat{\kappa}^{2}, & \#\{\partial A\cap\mathcal{B}_{+}^{\sigma}\}=2,  \\
					\kappa^{\#\{A\}-1}\hat{\kappa}, & \#\{\partial A\cap\mathcal{B}_{+}^{\sigma}\}=1,  \\
					\kappa^{\#\{A\}-1}, & \partial A\subset\mathcal{B}_{-}^{\sigma},\#\{\partial A\}\neq0,\\
					(N\hat{\kappa}+\kappa)\kappa^{N-1}, & \tilde{\mathcal{C}}^{\sigma} \tilde{\mathcal{D}}^{\sigma}=0, \#\{A\}=N,
				\end{cases}
			\end{split}
		\end{equation}
		where $\kappa=1-e^{-2\hat{J}}$ and $\hat{\kappa}=e^{-2\hat{J}}$.
		Moreover, $P_{\text{W}}(\sigma,\sigma^{A})$ admits a more detailed representation in terms of the connected components
		\begin{align}\label{eq:wolfftransition}
			P_{\text{W}}(\sigma,\sigma^{A})
			=\frac{\#\{A\}}{N}
			\begin{cases}
				1, 
				& \exists r:\ A\subset\mathcal{K}_r^\sigma,\ 
				\#\{\hat{\partial}\mathcal{K}_r^\sigma\}=1,\\[4pt]
				\kappa^{\#\{A\}-1}\hat{\kappa}^{2}, 
				& \exists r:\ A\subset\mathcal{K}_r^\sigma,\ 
				\#\{\hat{\partial}\mathcal{K}_r^\sigma\}=2,\ \#\{A\cap\hat{\partial}\mathcal{K}_{r}^{\sigma}\}=0,\\[4pt]
				\kappa^{\#\{A\}-1}\hat{\kappa}, 
				& \exists r:\ A\subset\mathcal{K}_r^\sigma,\ 
				\#\{\hat{\partial}\mathcal{K}_r^\sigma\}=2,\
				\#\{A\cap\hat{\partial}\mathcal{K}_{r}^{\sigma}\}=1,\\[4pt]
				\kappa^{\#\{A\}-1}, 
				& \exists r:\ A\subset\mathcal{K}_r^\sigma,\
				\#\{A\cap\hat{\partial}\mathcal{K}_{r}^{\sigma}\}=2,\\[4pt]
				\kappa^{\#\{A\}-1}\hat{\kappa}^{2}, 
				& \tilde{\mathcal C}^\sigma\tilde{\mathcal D}^\sigma=0,\ \#\{A\}\leq N-1,\\[4pt]
				(N\hat{\kappa}+\kappa)\kappa^{N-1}, 
				& \tilde{\mathcal C}^\sigma\tilde{\mathcal D}^\sigma=0,\ \#\{A\}=N,
			\end{cases}
		\end{align}	
		where
		\begin{align*}
			\mathcal{K}_r^\sigma\in\{\mathcal{C}_{r}^{\sigma},\mathcal{D}_{r}^{\sigma}\}.
		\end{align*}
		If $A$ is not connected, or if $A\subset[N]$ fails to satisfy $\sigma_i\sigma_j=+1$ for all $i,j\in A$, then 
		\begin{align*}
			P_{\text{W}}(\sigma,\sigma^{A})=\mathbb{P}(Y_{2}=\sigma^{A}\mid Y_{1}=\sigma)=0.
		\end{align*}
		In addition, for $\hat{J}=0$, we have $P_{\text{W}}(\sigma,\sigma^{A})=0$ for all $\#\{A\}\geq 2$ and for $A=\{i\}\subset[N]$, the transition probability reduces to
		\begin{equation}\label{eq:transition0}
			P_{\text{W}}(\sigma,\sigma^{i}):=P_{\text{W}}(\sigma,\sigma^{\{i\}})
			=\frac{1}{N}.
		\end{equation}
		Hence, when $\hat{J}=0$, the Wolff dynamics reduces to the simple random walk on the hypercube $\Omega=\{-1,+1\}^{N}$. 
		\item 
		For $\hat{J}\in [0,+\infty)$, the transition kernel $P_{\text{W}}$ satisfies the following detailed balance condition
		\begin{align*}	
			\mu_{N}(\eta)P_{\text{W}}(\eta,\sigma)=\mu_{N}(\sigma)P_{\text{W}}(\sigma,\eta),\qquad \eta,\sigma\in\Omega.
		\end{align*}
	\end{enumerate}
\end{theorem}
\begin{proof}
	Observe that steps 2, 3 and 4 of the algorithm ensure that a cluster is constructed by recursively adding only those nearest-neighbor sites whose spins are aligned with the current cluster. Therefore, if $A$ is disconnected, or if there exist $i,j\in A$ such that $\sigma_i\sigma_j=-1$, then no cluster with vertex set $A$ can be formed. Consequently,
	\begin{align*}
		P_{\text{W}}(\sigma,\sigma^{A})=\mathbb{P}(Y_{2}=\sigma^{A}\mid Y_{1}=\sigma)=0.
	\end{align*}	
	Now let $A\subset[N]$ be connected and satisfy $\sigma_i\sigma_j=+1$ for all $i,j\in A$. The probability that the initial site selected in step 1 belongs to $A$ is $\frac{\#\{A\}}{N}$. Conditioning on this event, if $\#\{\partial A\cap\mathcal{B}_{+}^{\sigma}\}\geq1$, the probability of constructing a cluster with vertex set exactly $A$ is given by the product of the inclusion probabilities along edges inside $A$ and the exclusion probabilities along edges in $\partial A\cap\mathcal{B}_{+}^{\sigma}$. Since, by step 5, the constructed cluster is flipped with probability 1, we obtain
	\begin{align*}
		P_{\text{W}}(\sigma,\sigma^{A})=\frac{\#\{A\}}{N}\kappa^{\#\{A\}-1}\hat{\kappa}^{\#\{\partial A\cap\mathcal{B}_{+}^{\sigma}\}}.
	\end{align*}
	If $\#\{\partial A\cap\mathcal{B}_{+}^{\sigma}\}=0$ and $\#\{\partial A\}\neq0$, then we have $\partial A\subset\mathcal{B}_{-}^{\sigma}$. In this case, all exclusion probabilities equal 1, and the cluster is determined solely by inclusion along edges inside $A$. Hence,
	\begin{align*}
		P_{\text{W}}(\sigma,\sigma^{A})=\frac{\#\{A\}}{N}\kappa^{\#\{A\}-1}.
	\end{align*}
	We finally consider the case $\partial A=\varnothing$ (i.e.,$A=[N]$). If $\sigma\notin\{-\textbf{1},+\textbf{1}\}$, then no cluster with vertex set $A$ can be formed, and thus
	\begin{align*}
		P_{\text{W}}(\sigma,\sigma^{A})=0.
	\end{align*}	
	If $\sigma\in\{-\textbf{1},+\textbf{1}\}$, due to the periodic boundary conditions, the probability of building the cluster with vertex set $A$ consists of two parts. The first term corresponds to the product over all sites in $A$ of their inclusion probabilities, multiplied by the exclusion probability associated with exactly one edge in $\mathcal{G}$. Since there are $N$ possible choices for this distinguished edge, this term is
	\begin{align*}
		N\kappa^{N-1}\hat{\kappa}.
	\end{align*}	
	The second term accounts for the case where all sites are included and no edge contributes an exclusion factor, yielding
	\begin{align*}
		\kappa^{N}.
	\end{align*}	
	Therefore, for $A=[N]$ and $\sigma\in\{-\textbf{1},+\textbf{1}\}$,
	\begin{align*}
		P_{\text{W}}(\sigma,\sigma^{A})=N\kappa^{N-1}\hat{\kappa}+\kappa^{N}.
	\end{align*}	
	When $\hat{J}=0$, we have $\kappa=0$, so the connection probability $\tilde{p}_{ij}=0$ for every edge $(i,j)\in \mathcal{G}$ in step 2 of the Wolff dynamics. Consequently, the cluster never grows beyond the initially selected site $i\in [N]$. Hence  $P_{\text{W}}(\sigma,\sigma^{A})=0$ for all $A$ with $\#\{A\}\geq 2$, and for $i\in [N]$
	\begin{align*}
		P_{\text{W}}(\sigma,\sigma^{i})
		=\frac{1}{N}.
	\end{align*}
	Collecting the above cases yields Equations~\eqref{eq:wolffedgecluster} and~\eqref{eq:transition0}. Moreover, these subsets can be described explicitly in terms of the connected components associated with $\sigma$, which leads directly to Equation~\eqref{eq:wolfftransition}.
	
	We now verify the detailed balance condition. For $A=[N]$ and $\sigma\in\{-\mathbf{1},+\mathbf{1}\}$, we have $\sigma^{A}\in\{-\mathbf{1},+\mathbf{1}\}$, and by Equation~\eqref{eq:wolffedgecluster},
	\begin{align*}
		\frac{P_{\text{W}}(\sigma,\sigma^{A})}{P_{\text{W}}(\sigma^{A},\sigma)}=\frac{N\kappa^{N-1}\hat{\kappa}+\kappa^{N}}{N\kappa^{N-1}\hat{\kappa}+\kappa^{N}}=1=\frac{\mu_{N}(\sigma^{A})}{\mu_{N}(\sigma)}.
	\end{align*}
	For a connected subset $A\subset[N]$, note that $\partial A\cap\mathcal{B}_{-}^{\sigma}=\partial A\cap\mathcal{B}_{+}^{\sigma^{A}}$. A direct computation gives
	\begin{align*}
		\frac{\mu_{N}(\sigma^{A})}{\mu_{N}(\sigma)}=&\frac{\exp(\hat{J}\sum_{i\sim j,i,j\notin A}\sigma_{i}\sigma_{j}+\hat{J}\sum_{i\sim j,i,j\in A}(-\sigma_{i})(-\sigma_{j})-\hat{J}\sum_{i\sim j,i\in A,j\notin A}\sigma_{i}\sigma_{j})}{\exp(\hat{J}\sum_{i\sim j,i,j\notin A}\sigma_{i}\sigma_{j}+\hat{J}\sum_{i\sim j,i,j\in A}\sigma_{i}\sigma_{j}+\hat{J}\sum_{i\sim j,i\in A,j\notin A}\sigma_{i}\sigma_{j})}\\
		=&\frac{\exp(-\hat{J}\sum_{i\sim j,i\in A,j\notin A}\sigma_{i}\sigma_{j})}{\exp(\hat{J}\sum_{i\sim j,i\in A,j\notin A}\sigma_{i}\sigma_{j})}
		=\exp(-2\hat{J}(\#\{\partial A\cap\mathcal{B}_{+}^{\sigma}\}-\#\{\partial A\cap\mathcal{B}_{-}^{\sigma}\}))\\
		=&\exp(-2\hat{J}(\#\{\partial A\cap\mathcal{B}_{+}^{\sigma}\}-\#\{\partial A\cap\mathcal{B}_{+}^{\sigma^{A}}\})).
	\end{align*}
	If $\#\{\partial A\cap\mathcal{B}_{+}^{\sigma}\}\geq1$, then 
	\begin{align*}
		\frac{P_{\text{W}}(\sigma,\sigma^{A})}{P_{\text{W}}(\sigma^{A},\sigma)}=&\frac{\frac{\#\{A\}}{N}\kappa^{\#\{A\}-1}\hat{\kappa}^{\#\{\partial A\cap\mathcal{B}_{+}^{\sigma}\}}}{\frac{\#\{A\}}{N}\kappa^{\#\{A\}-1}\hat{\kappa}^{\#\{\partial A\cap\mathcal{B}_{+}^{\sigma^{A}}\}}}
		=\exp(-2\hat{J}(\#\{\partial A\cap\mathcal{B}_{+}^{\sigma}\}-\#\{\partial A\cap\mathcal{B}_{+}^{\sigma^{A}}\}))\\
		=&\frac{\mu_{N}(\sigma^{A})}{\mu_{N}(\sigma)}.
	\end{align*}
	If $\#\{\partial A\}\neq0$ and $\partial A\subset\mathcal{B}_{-}^{\sigma}$, then $\partial A\subset\mathcal{B}_{+}^{\sigma^{A}}$, and
	\begin{align*}
		\frac{P_{\text{W}}(\sigma,\sigma^{A})}{P_{\text{W}}(\sigma^{A},\sigma)}=\frac{\frac{\#\{A\}}{N}\kappa^{\#\{A\}-1}}{\frac{\#\{A\}}{N}\kappa^{\#\{A\}-1}\hat{\kappa}^{\#\{\partial A\cap\mathcal{B}_{+}^{\sigma^{A}}\}}}=e^{2\hat{J}\#\{\partial A\}}=\frac{\mu_{N}(\sigma^{A})}{\mu_{N}(\sigma)}.
	\end{align*}
	This completes the proof of the detailed balance condition.
\end{proof}
\begin{remark}\label{rem:diffrentglauber}
	For the case $\hat{J}=0$, the transition probabilities of the Wolff dynamics reduce to Equation~\eqref{eq:transition0}, from which it follows that $P_{\text{W}}(\sigma,\sigma)=0$ for any $\sigma\in\Omega$. In contrast, for the Glauber dynamics, we have $P^{\text{GD}}(\sigma,\sigma)=1/2$ by Definition 2.12 in \cite{Chen2025document}. Indeed, as will be seen from Equation \eqref{eq:glauberinfinite} in the proof of Theorem \ref{thm:logsobolev}, at $\hat{J}=0$, the Glauber dynamics can be regarded as a lazy modification of the Wolff dynamics.
\end{remark}
\begin{remark}\label{rem:wolffdetailbalance} 
	The detailed balance condition for the Wolff dynamics was established in \cite{Wolff1989} for general $O(N)$ spin systems. 
	In the present 1D Ising setting, the transition probabilities in \eqref{eq:wolffedgecluster} can be verified directly to satisfy the detailed balance condition with respect to the Gibbs measure $\mu_{N}$. For completeness, we provide an explicit verification in the second part of Theorem~\ref{thm: transitionmatrix}.
\end{remark}
As observed by Wolff~\cite{Wolff1989}, the ergodicity of the dynamics follows from two key properties: first, there is always a non-vanishing probability that a cluster consists of only one site; second, there exists a reflection connecting any two configurations. We now give a rigorous verification of ergodicity for all $\hat{J}\in[0,+\infty)$. 
\begin{theorem}[Ergodicity for $\hat{J}\in[0,+\infty)$]\label{thm:birkhoff}
	Let $(Y_{k})_{k\in\mathbb{Z}^{+}}$ and $\mu_{N}$ be defined as in Theorem \ref{thm: transitionmatrix}. If the coupling constant $\hat{J}\in [0,+\infty)$, then for any initial configuration $Y_{1}\in\Omega$ and any real-valued function $g$ on $\Omega$, we have
	\begin{align*}	
		\lim_{M\rightarrow+\infty}\frac{1}{M}\sum_{k=1}^{M}g(Y_{k})=E_{\mu_{N}}[g], \qquad \mathbb{P}\text{-a.s.}.
	\end{align*}
\end{theorem}

\begin{proof}
	Recall that for any index set $A\subset [N]$, $\sigma^{A}$ denotes the configuration obtained from $\sigma\in\Omega$ by flipping the spins at the sites in $A$. For a given configuration $\tilde{\sigma}\in\Omega$ such that $\tilde{\sigma}\notin\{-\textbf{1},+\textbf{1}\}$, it follows from Equation \eqref{eq:wolfftransition} in Theorem \ref{thm: transitionmatrix} that for any $\hat{J}\in(0,+\infty)$,
	\begin{align*}
		P_{\text{W}}(\tilde{\sigma},\tilde{\sigma}^{\mathcal{D}_{k}^{\tilde{\sigma}}})=\mathbb{P}(Y_{2}=\tilde{\sigma}^{\mathcal{D}_{k}^{\tilde{\sigma}}}\mid Y_{1}=\tilde{\sigma})=\frac{\#\{\mathcal{D}_{k}^{\tilde{\sigma}}\}}{N}\kappa^{\#\{\mathcal{D}_{k}^{\tilde{\sigma}}\}-1}, \qquad 1\leq k\leq\tilde{\mathcal{D}}^{\tilde{\sigma}},
	\end{align*}
	and
	\begin{align*}
		P_{\text{W}}(\tilde{\sigma},\tilde{\sigma}^{\mathcal{C}_{k}^{\tilde{\sigma}}})=\mathbb{P}(Y_{2}=&\tilde{\sigma}^{\mathcal{C}_{k}^{\tilde{\sigma}}}\mid Y_{1}=\tilde{\sigma})=\frac{\#\{\mathcal{C}_{k}^{\tilde{\sigma}}\}}{N}\kappa^{\#\{\mathcal{C}_{k}^{\tilde{\sigma}}\}-1}, \qquad 1\leq k\leq\tilde{\mathcal{C}}^{\tilde{\sigma}}.
	\end{align*}
	Then for any $\tilde{\sigma}\neq +\textbf{1}$,
	\begin{align*}
		\mathbb{P}(Y_{\tilde{\mathcal{D}}^{\tilde{\sigma}}}=+\textbf{1}\mid Y_{1}=\tilde{\sigma})\geq\prod_{k=1}^{\tilde{\mathcal{D}}^{\tilde{\sigma}}}\left(\frac{\#\{\mathcal{D}_{k}^{\tilde{\sigma}}\}}{N}\kappa^{\#\{\mathcal{D}_{k}^{\tilde{\sigma}}\}-1}\right)>0.
	\end{align*}
	Similarly,
	\begin{align*}
		\mathbb{P}(Y_{\tilde{\mathcal{D}}^{\tilde{\sigma}}}=\tilde{\sigma}\mid Y_{1}=+\textbf{1})\geq\prod_{k=1}^{\tilde{\mathcal{D}}^{\tilde{\sigma}}}\left(\frac{\#\{\mathcal{D}_{k}^{\tilde{\sigma}}\}}{N}\kappa^{\#\{\mathcal{D}_{k}^{\tilde{\sigma}}\}-1}\hat{\kappa}^{2}\right)>0.
	\end{align*}
	Note that
	\begin{align*}
		\mathbb{P}(Y_{3}=+\textbf{1}\mid Y_{1}=+\textbf{1})\geq&\mathbb{P}(Y_{3}=+\textbf{1}\mid Y_{2}=-\textbf{1})\mathbb{P}(Y_{2}=-\textbf{1}\mid Y_{1}=+\textbf{1})\\
		=&P_{\text{W}}(+\textbf{1},-\textbf{1})P_{\text{W}}(-\textbf{1},+\textbf{1})=(N\hat{\kappa}+\kappa)^{2}\kappa^{2N-2}>0,
	\end{align*}
	and
	\begin{align*}
		\mathbb{P}(Y_{4}=+\textbf{1}\mid Y_{1}=+\textbf{1})\geq&P_{\text{W}}(+\textbf{1},(+\textbf{1})^{\{2\}})P_{\text{W}}((+\textbf{1})^{\{2\}},-\textbf{1})P_{\text{W}}(-\textbf{1},+\textbf{1})\\
		=&\frac{1}{N}\hat{\kappa}^{2}\frac{N-1}{N}\kappa^{N-2}(N\hat{\kappa}+\kappa)\kappa^{N-1}>0.
	\end{align*}	
	Hence, $(Y_k)_{k\in\mathbb{Z}^{+}}$ is an irreducible and aperiodic Markov chain on $\Omega$ for $\hat{J}\in(0,+\infty)$.
	
	For $\hat{J}=0$, Equation~\eqref{eq:transition0}~gives
	\begin{align*}
		P_{\text{W}}(\sigma,\sigma^{i})
		=\frac{1}{N},\qquad i\in[N],
	\end{align*}
	which implies $(Y_k)_{k\in\mathbb{Z}^{+}}$ is an irreducible Markov chain on $\Omega$ at $\hat{J}=0$. Since $\Omega$ is finite, $(Y_k)_{k\in\mathbb{Z}^{+}}$ is positive recurrent for all $\hat{J}\in[0,+\infty)$. Therefore, by Theorem C.1 in \cite{Levin2017document},
	\begin{align*}
		\frac{1}{M}\sum_{k=1}^{M}g(Y_{k})\rightarrow E_{\mu_{N}}[g],\qquad \mathbb{P}\text{-a.s.},
	\end{align*}
	where $g$ is a real-valued function defined on $\Omega$.
\end{proof}
In contrast, for $\hat{J}=+\infty$, the Wolff dynamics converges almost surely to the set $\{-\textbf{1},+\textbf{1}\}$. Starting from any initial configuration $Y_{1}\in\Omega$, the Markov chain reaches one of these configurations in finitely many steps and subsequently oscillates deterministically between them. Moreover, the first hitting time of the set $\{-\textbf{1},+\textbf{1}\}$ can be characterized explicitly in terms of the number of connected components in the decomposition of the initial configuration $Y_{1}$.
\begin{theorem}[Dynamics at $\hat{J}=+\infty$]\label{thm:infiniteJ}
	For $\hat{J}=+\infty$ (critical point), the Wolff dynamics reaches the set $\{-\textbf{1},+\textbf{1}\}$ in finitely many steps almost surely. Moreover, once the chain enters $\{-\textbf{1},+\textbf{1}\}$, it remains in this set thereafter. More precisely, if the initial configuration $Y_{1}\in\{-\textbf{1},+\textbf{1}\}$, then for any $k\geq 1$
	\begin{equation*}
		Y_{k}=
		\begin{cases}
			Y_{1}, & k \ \text{is odd},\\
			-Y_{1}, & k \ \text{is even}.
		\end{cases}
	\end{equation*}
	Furthermore, for any initial configuration $\tilde{\sigma}\in\Omega$ with $\tilde{\sigma}\notin\{-\textbf{1},+\textbf{1}\}$, we have
	\begin{align*}
		\mathbb{P}(Y_{\tilde{\mathcal{C}}^{\tilde{\sigma}}}\in\{-\textbf{1},+\textbf{1}\}\mid Y_{1}=\tilde{\sigma})=1,
	\end{align*}
	and consequently, for any $k\geq 0$,
	\begin{align*}
		\mathbb{P}(Y_{\tilde{\mathcal{C}}^{\tilde{\sigma}}+k}\in\{-\textbf{1},+\textbf{1}\}\mid Y_{1}=\tilde{\sigma})=1.
	\end{align*}	
\end{theorem}
\begin{proof}
	Recall that for $\hat{J}=+\infty$, the probability in step 2 of the Wolff dynamics satisfies $\tilde{p}_{ij}=1$ for any $(i,j)\in \mathcal{G}$. Hence, once a site $i$ is selected, the cluster constructed in the Wolff dynamics coincides with the connected component containing $i$.
	
	\textbf{Step 1: Deterministic dynamics on fully aligned configurations.} For the initial configuration $Y_{1}\in\{-\textbf{1},+\textbf{1}\}$, the vertex set of the cluster grown from any initial site $i\in[N]$ is always the entire set $[N]$. Consequently, for any $k\geq 1$
	\begin{align*}
		\mathbb{P}(Y_{k+1}=-\textbf{1}\mid Y_{k}=+\textbf{1})=1,\quad \mathbb{P}(Y_{k+1}=+\textbf{1}\mid Y_{k}=-\textbf{1})=1.
	\end{align*}
	Thus, for the initial configuration $Y_{1}=+\textbf{1}$,
	\begin{equation*}
		Y_{k}=\left\{\begin{array}{ll}
			+\textbf{1},& \qquad k \ \text{is odd},  \\
			-\textbf{1},& \qquad k \ \text{is even},
		\end{array}\right.
	\end{equation*}
	and for the initial configuration $Y_{1}=-\textbf{1}$,
	\begin{equation*}
		Y_{k}=\left\{\begin{array}{ll}
			-\textbf{1},& \qquad k \ \text{is odd},  \\
			+\textbf{1},& \qquad k \ \text{is even} .
		\end{array}\right.
	\end{equation*}
	Hence the chain alternates deterministically between $+\textbf{1}$ and $-\textbf{1}$.
	
	\textbf{Step 2: One-step reduction of connected components.}
	If the initial configuration $Y_{1}=\tilde{\sigma}\in\Omega$ satisfies $\tilde{\sigma}\notin\{-\textbf{1},+\textbf{1}\}$, then we have $\tilde{\mathcal{C}}^{\tilde{\sigma}}=\tilde{\mathcal{D}}^{\tilde{\sigma}}\neq 0$ in the decomposition~\eqref{eq:decom}~and
	\begin{align*}
		\mathcal{C}^{\tilde{\sigma}}=\bigcup_{r=1}^{\tilde{\mathcal{C}}^{\tilde{\sigma}}}\mathcal{C}_{r}^{\tilde{\sigma}},\qquad 	\mathcal{D}^{\tilde{\sigma}}=\bigcup_{r=1}^{\tilde{\mathcal{C}}^{\tilde{\sigma}}}\mathcal{D}_{r}^{\tilde{\sigma}}.
	\end{align*}
	Because $\hat{J}=+\infty$, the vertex set of the cluster grown from any initial site $i\in\mathcal{C}_{r}^{\tilde{\sigma}}$ (resp. $\mathcal{D}_{r}^{\tilde{\sigma}}$) is exactly the entire component $\mathcal{C}_{r}^{\tilde{\sigma}}$ (resp. $\mathcal{D}_{r}^{\tilde{\sigma}}$). Therefore, 
	\begin{align*}
		\sum_{r=1}^{\tilde{\mathcal{C}}^{\tilde{\sigma}}}\left(\mathbb{P}(Y_{2}=\tilde{\sigma}^{\mathcal{C}_{r}^{\tilde{\sigma}}}\mid Y_{1}=\tilde{\sigma})+ \mathbb{P}(Y_{2}=\tilde{\sigma}^{\mathcal{D}_{r}^{\tilde{\sigma}}}\mid Y_{1}=\tilde{\sigma})\right)=1.
	\end{align*}
	Define
	\begin{align*}
		\mathcal{O}_{1}^{\tilde{\sigma}}:=\bigcup_{r=1}^{\tilde{\mathcal{C}}^{\tilde{\sigma}}}\big(\{\eta\in\Omega\mid\eta=\tilde{\sigma}^{\mathcal{C}_{r}^{\tilde{\sigma}}}\}\cup\{\eta\in\Omega\mid\eta=\tilde{\sigma}^{\mathcal{D}_{r}^{\tilde{\sigma}}}\}\big).
	\end{align*}
	Then 
	\begin{align*}
		\mathbb{P}(Y_{2}\in\mathcal{O}_{1}^{\tilde{\sigma}}\mid Y_{1}=\tilde{\sigma})=1,
	\end{align*}
	where for any $A\subset\Omega$
	\begin{align*}
		\mathbb{P}(Y_{2}\in A\mid Y_{1}=\tilde{\sigma}):=\sum_{\eta\in A}\mathbb{P}(Y_{2}=\eta\mid Y_{1}=\tilde{\sigma}).
	\end{align*}
	
	\textbf{Step 3: Iteration.}
	For a configuration $\eta\in\mathcal{O}_{1}^{\tilde{\sigma}}$, $\mathcal{C}^{\eta}$ and $\mathcal{D}^{\eta}$
	can also be decomposed into a union of disjoint connected components
	\begin{align*}
		\mathcal{C}^{\eta}=\bigcup_{r=1}^{\tilde{\mathcal{C}}^{\eta}}\mathcal{C}_{r}^{\eta},\qquad 	\mathcal{D}^{\eta}=\bigcup_{r=1}^{\tilde{\mathcal{D}}^{\eta}}\mathcal{D}_{r}^{\eta},
	\end{align*}
	and $\tilde{\mathcal{C}}^{\eta}=\tilde{\mathcal{D}}^{\eta}=\tilde{\mathcal{C}}^{\tilde{\sigma}}-1$ for all $\eta\in\mathcal{O}_{1}^{\tilde{\sigma}}$. Similarly, for all $\eta\in\mathcal{O}_{1}^{\tilde{\sigma}}$, we have
	\begin{align*}
		\sum_{r=1}^{\tilde{\mathcal{C}}^{\tilde{\sigma}}-1}\left(\mathbb{P}(Y_{3}=\eta^{\mathcal{C}_{r}^{\eta}}\mid Y_{2}=\eta)+ \mathbb{P}(Y_{3}=\eta^{\mathcal{D}_{r}^{\eta}}\mid Y_{2}=\eta)\right)=1.
	\end{align*}
	Define 
	\begin{align*}
		\mathcal{O}_{2}^{\tilde{\sigma}}:=\bigcup_{\eta\in\mathcal{O}_{1}^{\tilde{\sigma}}}\bigcup_{r=1}^{\tilde{\mathcal{C}}^{\tilde{\sigma}}-1}\big(\{\xi\in\Omega\mid\xi=\eta^{\mathcal{C}_{r}^{\eta}}\}\cup\{\xi\in\Omega\mid\xi=\eta^{\mathcal{D}_{r}^{\eta}}\}\big).
	\end{align*}
	Then
	\begin{align*}
		\mathbb{P}(Y_{3}\in\mathcal{O}_{2}^{\tilde{\sigma}}\mid Y_{2}=\eta)=1,
	\end{align*} 
	for all $\eta\in\mathcal{O}_{1}^{\tilde{\sigma}}$, and by the Markov property
	\begin{align*}
		\mathbb{P}(Y_{3}\in\mathcal{O}_{2}^{\tilde{\sigma}}\mid Y_{1}=\tilde{\sigma})=&\sum_{\eta\in\Omega} \mathbb{P}(Y_{3}\in\mathcal{O}_{2}^{\tilde{\sigma}}\mid Y_{2}=\eta,Y_{1}=\tilde{\sigma})\mathbb{P}(Y_{2}=\eta\mid Y_{1}=\tilde{\sigma})\\
		=&\sum_{\eta\in\mathcal{O}_{1}^{\tilde{\sigma}}} \mathbb{P}(Y_{3}\in\mathcal{O}_{2}^{\tilde{\sigma}}\mid Y_{2}=\eta)\mathbb{P}(Y_{2}=\eta\mid Y_{1}=\tilde{\sigma})\\
		=&\mathbb{P}(Y_{2}\in\mathcal{O}_{1}^{\tilde{\sigma}}\mid Y_{1}=\tilde{\sigma})=1.
	\end{align*}
	For a configuration $\xi\in\mathcal{O}_{2}^{\tilde{\sigma}}$, $\mathcal{C}^{\xi}$ and $\mathcal{D}^{\xi}$
	can also be decomposed into the union of disjoint connected components as above and $\tilde{\mathcal{C}}^{\xi}=\tilde{\mathcal{D}}^{\xi}=\tilde{\mathcal{C}}^{\tilde{\sigma}}-2$ for all $\xi\in\mathcal{O}_{2}^{\tilde{\sigma}}$. 
	
	Iterating the above argument until the condition $\tilde{\mathcal{C}}^{\zeta}=\tilde{\mathcal{D}}^{\zeta}=1$  holds for all $\zeta\in\mathcal{O}_{\tilde{\mathcal{C}}^{\tilde{\sigma}}-1}^{\tilde{\sigma}}$, where
	\begin{align*}
		\mathcal{O}_{\tilde{\mathcal{C}}^{\tilde{\sigma}}-1}^{\tilde{\sigma}}:=\bigcup_{\vartheta\in\mathcal{O}_{\tilde{\mathcal{C}}^{\tilde{\sigma}}-2}^{\tilde{\sigma}}}\big(\{\zeta\in\Omega\mid\zeta=\vartheta^{\mathcal{C}_{1}^{\vartheta}}\}\cup\{\zeta\in\Omega\mid\zeta=\vartheta^{\mathcal{D}_{1}^{\vartheta}}\}\big),
	\end{align*}
	and 
	\begin{align*}
		\vartheta^{\mathcal{C}_{1}^{\vartheta}}=-\textbf{1},\qquad\vartheta^{\mathcal{D}_{1}^{\vartheta}}=+\textbf{1}.
	\end{align*}
	Then we have
	\begin{align*}
		\mathbb{P}(Y_{\tilde{\mathcal{C}}^{\tilde{\sigma}}-1}\in\mathcal{O}_{\tilde{\mathcal{C}}^{\tilde{\sigma}}-2}^{\tilde{\sigma}}\mid Y_{1}=\tilde{\sigma})=1,
	\end{align*}
	and for any $\varsigma\in\mathcal{O}_{\tilde{\mathcal{C}}^{\tilde{\sigma}}-2}^{\tilde{\sigma}}$,
	\begin{align*}
		\mathbb{P}(Y_{\tilde{\mathcal{C}}^{\tilde{\sigma}}}=+\textbf{1}\mid Y_{\tilde{\mathcal{C}}^{\tilde{\sigma}}-1}=\varsigma)+ \mathbb{P}(Y_{\tilde{\mathcal{C}}^{\tilde{\sigma}}}=-\textbf{1}\mid Y_{\tilde{\mathcal{C}}^{\tilde{\sigma}}-1}=\varsigma)=1.
	\end{align*}
	Consequently,
	\begin{align*}
		\mathbb{P}(Y_{\tilde{\mathcal{C}}^{\tilde{\sigma}}}=+\textbf{1}\mid Y_{1}=\tilde{\sigma})+\mathbb{P}(Y_{\tilde{\mathcal{C}}^{\tilde{\sigma}}}=-\textbf{1}\mid Y_{1}=\tilde{\sigma})=1.
	\end{align*}
	This completes the proof.
\end{proof}
\begin{remark}\label{rem:glauberfiniteflip}
	The theorem shows that $\{-\textbf{1},+\textbf{1}\}$ can be viewed as a pair of globally attractive configurations for the Wolff dynamics at $\hat{J}=+\infty$. Once this set is reached, the dynamics becomes deterministic and alternates between the two configurations, thus forming a periodic orbit of period $2$. Moreover, the hitting time depends explicitly on the number of connected components in the decomposition of the initial configuration. In contrast, such behavior is absent for the Glauber dynamics at $\hat{J}=+\infty$.
\end{remark}

\subsection{The log-Sobolev inequality for the Wolff dynamics}\label{dec:logsobolevwolff}
The log-Sobolev inequality plays a central role in the analysis of spin models. On the one hand, it is a key tool in establishing the uniqueness of the infinite-volume Gibbs measures for such models (see e.g. \cite{Menz2014document1,Bauerschmidt2021,Bauerschmidt2024document1}). On the other hand, it implies exponential convergence of the associated dynamics to equilibrium, with a rate governed by the log-Sobolev constant, and a detailed discussion can be found in \cite{Royer2007document}. In this section, we establish the log-Sobolev inequality as well as the Poincar\'e inequality for the Wolff dynamics in the entire subcritical regime $\hat{J}\in[0,+\infty)$ of the 1D Ising model. Our approach combines a comparison argument with the work of 
Bauerschmidt and Dagallier \cite{Bauerschmidt2024document} on the 
log-Sobolev inequality for the dynamics whose Dirichlet form coincides 
with Equation (3.6) in \cite{Martinelli1997document}. These results yield quantitative estimates on the ergodic averages for the Wolff dynamics.
\begin{theorem}\label{thm:logsobolev}
	Let $(Y_k)_{k\in\mathbb{Z}^+}$ be the Wolff dynamics with invariant measure $\mu_N$. 
	For any $\hat{J}\in[0,+\infty)$, the log-Sobolev inequality
	\begin{align}\label{eq:logsobolev} 
		E_{\mu_{N}}[f^{2}\log f^{2}]-E_{\mu_{N}}[f^{2}]\log E_{\mu_{N}}[f^{2}]\leq&e^{2\hat{J}}(e^{4\hat{J}}+1)\big(\frac{1}{2}+\hat{J}e^{\frac{e^{2\hat{J}}-1}{2}}\big)N	\mathcal{E}^{\text{W}}_{\mu_{N}}(f),
	\end{align}
	holds for all functions $f:\Omega\to\mathbb{R}$, where
	\begin{align*}	
		\mathcal{E}^{\text{W}}_{\mu_{N}}(f):=\frac{1}{2}\sum_{\sigma,\eta\in\Omega}[f(\sigma)-f(\eta)]^{2}P_{\text{W}}(\eta,\sigma)\mu_{N}(\eta).
	\end{align*}
	%
\end{theorem}
\begin{proof}
	We compare the Dirichlet form of the Wolff dynamics with that of the Glauber dynamics. According to Equation 3.12 in \cite{Levin2017document}, for the 1D Ising model, the transition probabilities of the Glauber dynamics are given by
	\begin{align*}
		P^{\text{GD}}(\sigma,\eta)=\frac{1}{N}\sum_{i=1}^{N}\frac{e^{\hat{J}\eta_{i}(\sigma_{i-1}+\sigma_{i+1})}}{e^{\hat{J}\eta_{i}(\sigma_{i-1}+\sigma_{i+1})}+e^{-\hat{J}\eta_{i}(\sigma_{i-1}+\sigma_{i+1})}}\mathbb{I}_{\{\eta_{j}=\sigma_{j}\ \text{for}\ j\neq i\}},
	\end{align*}
	where the notation $\mathbb{I}_{A}$ represents the indicator function of a set $A$. In particular, $P^{\mathrm{GD}}(\sigma,\sigma^A)=0$ unless $A=\{i\}$, and
	\begin{align}\label{eq:glaubertrans}
		P^{\text{GD}}(\sigma,\sigma^{i})=\frac{1}{N}\frac{e^{-\hat{J}\sigma_{i}(\sigma_{i-1}+\sigma_{i+1})}}{e^{\hat{J}(\sigma_{i-1}+\sigma_{i+1})}+e^{-\hat{J}(\sigma_{i-1}+\sigma_{i+1})}}.
	\end{align}
	Recalling the expression of $P_{\text{W}}(\sigma,\eta)$ in Equation~\eqref{eq:wolfftransition} for the Wolff dynamics, we have
	\begin{align}\label{eq:simglewolff}
		P_{\text{W}}(\sigma,\sigma^{i})
		=\frac{1}{N}
		\begin{cases}
			1, 
			& \exists r:\ \{i\}=\mathcal{K}_r^\sigma,\\[4pt]
			e^{-4\hat J}, 
			& \exists r:\ i\in\mathcal{K}_r^\sigma,\ 
			\#\{\hat{\partial}\mathcal{K}_r^\sigma\}=2,\
			i\notin \hat{\partial}\mathcal{K}_r^\sigma,\\[4pt]
			e^{-2\hat J}, 
			& \exists r:\ i\in\mathcal{K}_r^\sigma,\ 
			\#\{\hat{\partial}\mathcal{K}_r^\sigma\}=2,\
			i\in \hat{\partial}\mathcal{K}_r^\sigma,\\[4pt]
			e^{-4\hat J}, 
			& \tilde{\mathcal C}^\sigma\tilde{\mathcal D}^\sigma=0,
		\end{cases}
	\end{align}	
	where
	\begin{align*}
		\mathcal{K}_r^\sigma\in\{\mathcal{C}_{r}^{\sigma},\mathcal{D}_{r}^{\sigma}\}.
	\end{align*}
	For comparison with the Wolff dynamics, we rewrite the transition probabilities of Glauber dynamics in terms of the cluster decomposition introduced above. By Equation~\eqref{eq:glaubertrans}, we have
	\begin{align}\label{eq:glauberinfinite}
		P^{\mathrm{GD}}(\sigma,\sigma^{i})
		=\frac{1}{N}
		\begin{cases}
			\frac{e^{2\hat{J}}}{e^{2\hat{J}}+e^{-2\hat{J}}}, 
			& \exists r:\ \{i\}=\mathcal{K}_r^\sigma,\\[4pt]
			\frac{e^{-2\hat{J}}}{e^{2\hat{J}}+e^{-2\hat{J}}}, 
			& \exists r:\ i\in\mathcal{K}_r^\sigma,\ 
			\#\{\hat{\partial}\mathcal{K}_r^\sigma\}=2,\
			i\notin \hat{\partial}\mathcal{K}_r^\sigma,\\[4pt]
			\frac{1}{2}, 
			& \exists r:\ i\in\mathcal{K}_r^\sigma,\ 
			\#\{\hat{\partial}\mathcal{K}_r^\sigma\}=2,\
			i\in \hat{\partial}\mathcal{K}_r^\sigma,\\[4pt]
			\frac{e^{-2\hat{J}}}{e^{2\hat{J}}+e^{-2\hat{J}}}, 
			& \tilde{\mathcal C}^\sigma\tilde{\mathcal D}^\sigma=0.
		\end{cases}
	\end{align}	
	A case-by-case comparison of \eqref{eq:glauberinfinite} and \eqref{eq:simglewolff} shows that, for all $\sigma$ and $i$,
	\begin{align}\label{eq:comparisononsigle}	
		P^{\text{GD}}(\sigma,\sigma^{i})\leq\frac{1}{2}e^{2\hat{J}} P_{\text{W}}(\sigma,\sigma^{i}).
	\end{align}
	Since $P^{\text{GD}}(\sigma,\eta)=0$ unless $\eta=\sigma^i$ for some $i$, we have, for all $\sigma\neq\eta$,
	\begin{align*}	
		P^{\text{GD}}(\sigma,\eta)\leq \frac{1}{2}e^{2\hat{J}} P_{\text{W}}(\sigma,\eta).
	\end{align*}
	Therefore, restricting the Dirichlet form of the Wolff dynamics to single-site flips yields
	\begin{align}\label{eq:compariondicform}	
		2\mathcal{E}^{\text{W}}_{\mu_{N}}(f)=&\sum_{\sigma,\eta\in\Omega}[f(\sigma)-f(\eta)]^{2}P_{\text{W}}(\eta,\sigma)\mu_{N}(\eta)\nonumber\\
		=&\sum_{\substack{\sigma,\eta\in\Omega, \sigma\neq\eta}}[f(\sigma)-f(\eta)]^{2}P_{\text{W}}(\eta,\sigma)\mu_{N}(\eta)\nonumber\\
		\geq&\sum_{i=1}^{N}\sum_{\eta\in\Omega}[f(\eta^{i})-f(\eta)]^{2}P_{\text{W}}(\eta,\eta^{i})\mu_{N}(\eta)\\
		\geq&2e^{-2\hat{J}}\sum_{i=1}^{N}\sum_{\eta\in\Omega}[f(\eta^{i})-f(\eta)]^{2}P^{\text{GD}}(\eta,\eta^{i})\mu_{N}(\eta)\nonumber\\
		=&2e^{-2\hat{J}}\sum_{\sigma,\eta\in\Omega}[f(\sigma)-f(\eta)]^{2}P^{\text{GD}}(\eta,\sigma)\mu_{N}(\eta)=4e^{-2\hat{J}}	\mathcal{E}^{\text{GD}}_{\mu_{N}}(f),\nonumber
	\end{align}
	where in the second inequality we used \eqref{eq:comparisononsigle}, and
	\begin{align*}	
		\mathcal{E}^{\text{GD}}_{\mu_{N}}(f):=\frac{1}{2}\sum_{\sigma,\eta\in\Omega}[f(\sigma)-f(\eta)]^{2}P^{\text{GD}}(\eta,\sigma)\mu_{N}(\eta),
	\end{align*}
	denotes the Dirichlet form associated with the Glauber dynamics. From Equations~\eqref{eq:glaubertrans}~and~\eqref{eq:glauberinfinite}, we have
	\begin{align*} 
		2\mathcal{E}^{\text{GD}}_{\mu_{N}}(f)=&\sum_{\sigma,\eta\in\Omega}[f(\sigma)-f(\eta)]^{2}P^{\text{\text{GD}}}(\eta,\sigma)\mu_{N}(\eta)\\
		=&\sum_{\eta\in\Omega}\sum_{i=1}^{N}[f(\eta^{i})-f(\eta)]^{2}P^{\text{\text{GD}}}(\eta,\eta^{i})\mu_{N}(\eta)\\
		=&\frac{1}{N}\sum_{\eta\in\Omega}\sum_{i=1}^{N}[f(\eta^{i})-f(\eta)]^{2}\frac{e^{-\hat{J}\eta_{i}(\eta_{i-1}+\eta_{i+1})}}{e^{\hat{J}(\eta_{i-1}+\eta_{i+1})}+e^{-\hat{J}(\eta_{i-1}+\eta_{i+1})}}\mu_{N}(\eta)\\
		\geq&\frac{1}{e^{4\hat{J}}+1}\frac{1}{N}\sum_{i=1}^{N}\sum_{\eta\in\Omega}[f(\eta^{i})-f(\eta)]^{2}\mu_{N}(\eta),
	\end{align*}
	where
	\begin{align*} 
		\sum_{i=1}^{N}\sum_{\eta\in\Omega}[f(\eta^{i})-f(\eta)]^{2}\mu_{N}(\eta)=\sum_{i=1}^{N}E_{\mu_{N}}[(f(\eta^{i})-f(\eta))^{2}],
	\end{align*}
	is the Dirichlet form used in \cite{Bauerschmidt2024document} (i.e., Equation (3.6) in \cite{Martinelli1997document}). By Theorem 1.1 in \cite{Bauerschmidt2024document}, we obtain the following log-Sobolev inequality
	\begin{align}\label{eq:rolandlogsobolev} 
		E_{\mu_{N}}[f^{2}\log f^{2}]-E_{\mu_{N}}[f^{2}]\log E_{\mu_{N}}[f^{2}]\leq&\big(\frac{1}{2}+\int_{0}^{\hat{J}}e^{\int_{0}^{t}\chi_{s}ds}dt\big)\sum_{i=1}^{N}E_{\mu_{N}}[(f(\eta)-f(\eta^{i}))^{2}]\nonumber\\
		\leq&e^{2\hat{J}}(e^{4\hat{J}}+1)\big(\frac{1}{2}+\int_{0}^{\hat{J}}e^{\int_{0}^{t}\chi_{s}ds}dt\big)N	\mathcal{E}^{\text{W}}_{\mu_{N}}(f),
	\end{align}
	where
	\begin{align*} 
		\chi_{\hat{J}}:=\sup_{i\in[N]}\sum_{j=1}^{N}E_{\mu_{N}}[\sigma_{i}\sigma_{j}].
	\end{align*}
	It remains to bound $\chi_{\hat{J}}$. Note that for $1\leq i<j\leq N$,
	\begin{align}\label{eq:1DIsingcorrela}	E_{\mu_{N}}[\sigma_{i}\sigma_{j}]=\frac{1}{Z_{N}}\sum_{\sigma_{1}=\pm 1}\cdots\sum_{\sigma_{N}=\pm 1}\sigma_{i}\sigma_{j}\exp\left(\hat{J}\sum_{l=1}^{N}\sigma_{l}\sigma_{l+1}\right)=\frac{\boldsymbol{\theta}^{j-i}+\boldsymbol{\theta}^{N-j+i}}{1+\boldsymbol{\theta}^{N}},
	\end{align}
	where the last equality follows from p.174 in \cite{Morandi2004document}. Therefore, for any $1\leq i\leq N$,
	\begin{align*}	\sum_{j=1}^{N}E_{\mu_{N}}[\sigma_{i}\sigma_{j}]&\leq 1+\sum_{j=i+1}^{N}\frac{\boldsymbol{\theta}^{j-i}+\boldsymbol{\theta}^{N-j+i}}{1+\boldsymbol{\theta}^{N}}+\sum_{j=1}^{i-1}\frac{\boldsymbol{\theta}^{i-j}+\boldsymbol{\theta}^{N-i+j}}{1+\boldsymbol{\theta}^{N}}.
	\end{align*}
	A direct computation shows that
	\begin{align*}
		&\sum_{j=i+1}^{N}\left(\boldsymbol{\theta}^{j-i}+\boldsymbol{\theta}^{N-j+i}\right)+\sum_{j=1}^{i-1}\left(\boldsymbol{\theta}^{i-j}+\boldsymbol{\theta}^{N-i+j}\right)\\
		=&\frac{\lambda_{+}}{\lambda_{+}-\lambda_{-}}\big(\boldsymbol{\theta}-\boldsymbol{\theta}^{N-i+1}+\boldsymbol{\theta}^{i}-\boldsymbol{\theta}^{N}+\boldsymbol{\theta}-\boldsymbol{\theta}^{i}+\boldsymbol{\theta}^{N-i+1}-\boldsymbol{\theta}^{N}\big)\\
		=&\frac{2\lambda_{+}}{\lambda_{+}-\lambda_{-}}(\boldsymbol{\theta}-\boldsymbol{\theta}^{N}).
	\end{align*}
	Hence,
	\begin{align}\label{eq:sumcorrelation}	\sum_{j=1}^{N}E_{\mu_{N}}[\sigma_{i}\sigma_{j}]&\leq 1+\frac{2\lambda_{+}}{\lambda_{+}-\lambda_{-}}\frac{\boldsymbol{\theta}-\boldsymbol{\theta}^{N}}{1+\boldsymbol{\theta}^{N}},
	\end{align}
	which implies
	\begin{equation}\label{eq:upperspeci} 
		\chi_{\hat{J}}\leq 1+\frac{2\lambda_{+}}{\lambda_{+}-\lambda_{-}}\frac{ \boldsymbol{\theta}- \boldsymbol{\theta}^{N}}{1+ \boldsymbol{\theta}^{N}}\leq 1+\frac{2\lambda_{+}}{\lambda_{+}-\lambda_{-}}\boldsymbol{\theta}=1+2\frac{\lambda_{-}}{\lambda_{+}-\lambda_{-}}=\frac{\lambda_{+}+\lambda_{-}}{\lambda_{+}-\lambda_{-}}=e^{2\hat{J}}.
	\end{equation}
	Putting Equation~\eqref{eq:upperspeci} into Equation~\eqref{eq:rolandlogsobolev} yields
	\begin{align*} 
		E_{\mu_{N}}[f^{2}\log f^{2}]-E_{\mu_{N}}[f^{2}]\log E_{\mu_{N}}[f^{2}]\leq&e^{2\hat{J}}(e^{4\hat{J}}+1)\big(\frac{1}{2}+\hat{J}e^{\frac{e^{2\hat{J}}-1}{2}}\big)N	\mathcal{E}^{\text{W}}_{\mu_{N}}(f).
	\end{align*}
\end{proof}
\begin{corollary}[\cite{Gross1975document}, Theorem 3 and \cite{Diaconis1996}, Example 3.2]\label{cor:logsobolevforrandonwaok}
	Let $(Y_{k})_{k\in\mathbb{Z}^{+}}$ and $\mu_{N}$ be defined as in Theorem \ref{thm: transitionmatrix}. At infinite temperature (i.e., $\hat{J}=0$), the Wolff dynamics reduces to the simple random walk on the hypercube $\Omega=\{-1,+1\}^{N}$, where at each step a single spin is chosen uniformly at random and flipped. Applying Theorem~\ref{thm:logsobolev} to the Markov chain with transition probability~\eqref{eq:transition0}, the log-Sobolev inequality
	\begin{align}\label{eq:logsobolevrandomwork} 
		E_{\mu_{N}}[f^{2}\log f^{2}]-E_{\mu_{N}}[f^{2}]\log E_{\mu_{N}}[f^{2}]\leq&N	\mathcal{E}^{\text{W},0}_{\mu_{N}}(f),
	\end{align}
	holds for all functions $f:\Omega\to\mathbb{R}$, where
	\begin{align*}	
		\mathcal{E}^{\text{W},0}_{\mu_{N}}(f):=\frac{1}{2N}\sum_{i=1}^{N}E_{\mu_{N}}\big[(f(\sigma)-f(\sigma^{i}))^{2}\big]=\frac{1}{2N}\sum_{i=1}^{N}(f(\sigma)-f(\sigma^{i}))^{2}\frac{1}{2^{N}}.
	\end{align*}
\end{corollary}
\begin{remark}
	As shown by Gross in Theorem 3 of \cite{Gross1975document}, the random walk on the two point space $\{-1,+1\}$ satisfies a log-Sobolev inequality with constant $C_{\text{LS}}=1$. Diaconis and Saloff-Coste extended Gross's result to the random walk on the hypercube $\Omega=\{-1,+1\}^{N}$ in Example 3.2 of \cite{Diaconis1996}, where they established a log-Sobolev inequality with constant $C_{\text{LS}}=N$. 
	
	In the special case $\hat{J}=0$, 
	Corollary~\ref{cor:logsobolevforrandonwaok} recovers these classical results. More generally, for $\hat{J}>0$, Theorem~\ref{thm:logsobolev}
	provides an upper bound on the log-Sobolev constant of the Wolff dynamics that grows at most linearly in the dimension $N$, up to a multiplicative factor depending on $\hat{J}$.
\end{remark}
The Poincar\'e inequality follows from Theorem~\ref{thm:logsobolev} via the standard implication that a log-Sobolev inequality implies a Poincar\'e inequality.
\begin{corollary}[The Poincar\'e inequality]\label{eq:Poincare}
	Let $(Y_{k})_{k\in\mathbb{Z}^{+}}$ and $\mu_{N}$ be defined as in Theorem \ref{thm: transitionmatrix}. For the coupling constant $\hat{J}\in [0,+\infty)$, the following Poincar\'e inequality holds 
	\begin{align*} 
		E_{\mu_{N}}[f^{2}]-(E_{\mu_{N}}[f])^{2}\leq&\frac{e^{2\hat{J}}(e^{4\hat{J}}+1)}{2}\big(\frac{1}{2}+\hat{J}e^{\frac{e^{2\hat{J}}-1}{2}}\big)N	\mathcal{E}^{\text{W}}_{\mu_{N}}(f),
	\end{align*}
	where $f$ is any real-valued function defined on $\Omega$.
\end{corollary}
Having the log-Sobolev inequality and the Poincar\'e inequality in hand, we are now in a position to give quantitative estimates for ergodic averages of the Wolff dynamics. The proof is based on a classical spectral decomposition argument (see e.g. \cite{Bassetti2006document,Rudolf2009document,Gong2006document}).
\begin{theorem}\label{thm: chernoffbound}
	Let $(Y_{k})_{k\in\mathbb{Z}^{+}}$ and $\mu_{N}$ be defined as in Theorem \ref{thm: transitionmatrix}. If the coupling constant $\hat{J}\in [0,+\infty)$, then for any real-valued function $f:\Omega\to\mathbb{R}$, we have
	\begin{align*}	
		\left\|\frac{1}{M}\sum_{k=1}^{M}P_{\text{W}}^{k}f-E_{\mu_{N}}[f]\right\|_{L^{2}(\mu_{N})}
		\leq\frac{\|f\|_{L^{2}(\mu_N)}}{M}\bigg(2+ \frac{e^{2\hat{J}}(e^{4\hat{J}}+1)}{2}\big(\frac{1}{2}+\hat{J}e^{\frac{e^{2\hat{J}}-1}{2}}\big)N\bigg),
	\end{align*}
	where $\|\cdot\|_{L^{2}(\mu_N)}$ denotes the norm associated with the inner product $\langle \cdot, \cdot \rangle_{\mu_N}$, and
	\begin{align*}
		P_{\text{W}}^{k}f(x)=\sum_{y\in\Omega}f(y) P^{k}_{\text{W}}(x,y),\qquad x\in\Omega, 1\leq k\leq M.
	\end{align*}
	Moreover,
	\begin{align*}
		E^{\mathbb{P}_{\mu_{N}}}\left[\left( \frac{1}{M}\sum_{k=1}^{M}f(Y_{k})-E_{\mu_{N}}[f]\right)^{2}\right]
		\leq\frac{\|f\|^{2}_{L^{2}(\mu_N)}}{M}e^{2\hat{J}}(e^{4\hat{J}}+1)\big(\frac{1}{2}+\hat{J}e^{\frac{e^{2\hat{J}}-1}{2}}\big)N,
	\end{align*}	
	where $\mathbb{P}_{\mu_{N}}$ denotes the law of the process started from $\mu_N$.
\end{theorem}
\begin{proof}
	Let $D_{\mu_N}$ denote the diagonal matrix indexed by $\Omega$, with entries
	\begin{align*}
		(D_{\mu_N})_{x,x} =\mu_N(x),\qquad x\in\Omega.
	\end{align*}
	Define
	\begin{align*}
		\tilde{P}_{\text{W}}(x,y):=\frac{\sqrt{\mu_{N}(x)}}{\sqrt{\mu_{N}(y)}}P_{\text{W}}(x,y),\qquad x,y\in\Omega.
	\end{align*}
	This corresponds to the similarity transform
	\begin{align*}
		\tilde{P}_{\text{W}} = D_{\mu_N}^{1/2}P_{\text{W}}D_{\mu_N}^{-1/2}.
	\end{align*}
	By the detailed balance condition,
	\begin{align*}
		\mu_{N}(x)P_{\text{W}}(x,y)=\mu_{N}(y)P_{\text{W}}(y,x),
	\end{align*}
	it follows that $\tilde{P}_{\text{W}}$ is symmetric, i.e.,
	\begin{align*}
		\tilde{P}_{\text{W}}(x,y)=\tilde{P}_{\text{W}}(y,x),\qquad x,y\in\Omega.
	\end{align*}
	Hence, by the spectral theorem, there exists an orthonormal basis $\{\psi_{j}\}_{j=1}^{\Lambda}$ and real eigenvalues $\{\lambda_{j}\}_{j=1}^{\Lambda}$ such that
	\begin{align*}
		\sum_{y\in\Omega}\tilde{P}_{\text{W}}(x,y)\psi_{j}(y)=\lambda_{j}\psi_{j}(x),
	\end{align*}
	and
	\begin{align*}
		\tilde{P}_{\text{W}}(x,y)=\sum_{j=1}^{\Lambda}\lambda_{j}\psi_{j}(x)\psi_{j}(y),
	\end{align*}
	where $\Lambda=\#\{\Omega\}=2^{N}$. Moreover,
	\begin{align*}
		\sum_{y\in\Omega}\tilde{P}_{\text{W}}(x,y)\sqrt{\mu_{N}(y)}=\sum_{y\in\Omega}\sqrt{\mu_{N}(x)}P_{\text{W}}(x,y)=\sqrt{\mu_{N}(x)},
	\end{align*}
	so that 1 is an eigenvalue of $\tilde{P}_{\text{W}}$. Define $\xi_j :=D_{\mu_N}^{-1/2}\psi_j$, then $\{\xi_j\}_{j=1}^{\Lambda}$ forms an orthonormal basis of $L^{2}(\mu_{N})$, and
	\begin{align*}
		\sum_{y\in\Omega}P_{\text{W}}(x,y)\xi_{j}(y)=&\sum_{y\in\Omega}P_{\text{W}}(x,y)\frac{\psi_{j}(y)}{\sqrt{\mu_{N}(y)}}\\
		=&\sum_{y\in\Omega}\frac{1}{\sqrt{\mu_{N}(x)}}\tilde{P}_{\text{W}}(x,y)\psi_{j}(y)=\frac{\lambda_{j}}{\sqrt{\mu_{N}(x)}}\psi_{j}(x)=\lambda_{j}\xi_{j}(x).
	\end{align*}
	Consequently, 
	\begin{align*}
		P_{\text{W}}(x,y)=\tilde{P}_{\text{W}}(x,y)\frac{\sqrt{\mu_{N}(y)}}{\sqrt{\mu_{N}(x)}}=\frac{\sqrt{\mu_{N}(y)}}{\sqrt{\mu_{N}(x)}}\sum_{j=1}^{\Lambda}\lambda_{j}\psi_{j}(x)\psi_{j}(y)=\sum_{j=1}^{\Lambda}\lambda_{j}\xi_{j}(x)\xi_{j}(y)\mu_{N}(y),
	\end{align*}
	and
	\begin{align*}
		P^{k}_{\text{W}}(x,y)=\sum_{j=1}^{\Lambda}\lambda^{k}_{j}\xi_{j}(x)\xi_{j}(y)\mu_{N}(y).
	\end{align*}
	Since $P_{\text{W}}$ is irreducible and reversible with respect to $\mu_N$, the eigenspace associated with eigenvalue $1$ is the 1D space spanned by $\xi_1 = \textbf{1}$, which is normalized in $L^2(\mu_N)$ since $\mu_{N}$ is a probability measure. The eigenvalues satisfy
	\begin{align*}
		1 = \lambda_1 > \lambda_2 \geq \cdots \geq \lambda_{\Lambda}\geq -1.
	\end{align*}
	Then for any real-valued function $f:\Omega\rightarrow\mathbb{R}$,
	\begin{equation}\label{eq:pfexpression}
		\begin{split}
			P_{\text{W}}^{k}f(x)=&\sum_{y\in\Omega}f(y) P^{k}_{\text{W}}(x,y)\\
			=&\xi_{1}(x)\sum_{y\in\Omega}f(y)\xi_{1}(y)\mu_{N}(y)+\sum_{j=2}^{\Lambda}\lambda^{k}_{j}\xi_{j}(x)\sum_{y\in\Omega}f(y)\xi_{j}(y)\mu_{N}(y)\\
			=&\langle\xi_{1}, f\rangle_{\mu_{N}}+\sum_{j=2}^{\Lambda}\lambda^{k}_{j}\langle f,\xi_{j}\rangle_{\mu_{N}}\xi_{j}(x),
		\end{split}
	\end{equation}
	where $\langle\cdot, \cdot\rangle_{\mu_{N}}$ is the inner product on $L^{2}(\mu_{N})$ and $\langle\xi_{1},f\rangle_{\mu_{N}}=E_{\mu_{N}}[f]$. Therefore,
	\begin{align*}
		\left\|\frac{1}{M}\sum_{k=1}^{M} P^{k}_{\text{W}}f-E_{\mu_{N}}[f]	\right\|^{2}_{L^{2}(\mu_{N})}=&\frac{1}{M^{2}}\sum_{i=2}^{\Lambda}\sum_{j=2}^{\Lambda}\big(\sum_{k=1}^{M}\lambda^{k}_{i}\big)\big(\sum_{k=1}^{M}\lambda^{k}_{j}\big)\langle f,\xi_{i}\rangle_{\mu_{N}}\langle f,\xi_{j}\rangle_{\mu_{N}}\langle\xi_{i},\xi_{j}\rangle_{\mu_{N}}\\
		=&\frac{1}{M^{2}}\sum_{i=2}^{\Lambda}\sum_{j=2}^{\Lambda}\big(\sum_{k=1}^{M}\lambda^{k}_{i}\big)\big(\sum_{k=1}^{M}\lambda^{k}_{j}\big)\langle f,\xi_{i}\rangle_{\mu_{N}}\langle f,\xi_{j}\rangle_{\mu_{N}}\delta_{ij}\\
		=&\frac{1}{M^{2}}\sum_{i=2}^{\Lambda}\big(\sum_{k=1}^{M}\lambda^{k}_{i}\big)^{2}\langle f,\xi_{i}\rangle_{\mu_{N}}^{2},
	\end{align*}
	where $\delta_{ij}$ denotes the Kronecker delta, i.e., $\delta_{ij}=1$ if $i=j$ and 0 otherwise. Define
	\begin{align*}
		\mathcal{C}(M,\lambda_{i}):=\sum_{k=1}^{M}\lambda^{k}_{i}.
	\end{align*}
	When $M$ is an odd integer, the function $\mathcal{C}(M,x)$ is an increasing function of $x$ on the interval $[-1,1]$. Consequently, for $2\leq i\leq\Lambda$ 
	\begin{align*}
		\mathcal{C}^{2}(M,\lambda_{i})\leq \max\{\mathcal{C}^{2}(M,\lambda_{2}),\mathcal{C}^{2}(M,\lambda_{\Lambda})\}.
	\end{align*}
	If $|\lambda_{\Lambda}|\leq\lambda_{2}$, then for $2\leq i\leq\Lambda$,
	\begin{align*}
		|\mathcal{C}(M,\lambda_{i})|\leq \sum_{k=1}^{M}|\lambda_{i}|^{k}\leq \sum_{k=1}^{M}\lambda_{2}^{k}=\mathcal{C}(M,\lambda_{2})=\frac{\lambda_{2}-\lambda_{2}^{M+1}}{1-\lambda_{2}}\leq \frac{1}{1-\lambda_{2}}.
	\end{align*}
	If $\lambda_{2}<|\lambda_{\Lambda}|\leq1$, then
	\begin{align*}
		|\mathcal{C}(M,\lambda_{2})|\leq&\max\{\frac{1}{1-\lambda_{2}},| \frac{\lambda_{2}-\lambda_{2}^{M+1}}{1-\lambda_{2}}|\}=\max\{\frac{1}{1-\lambda_{2}},\frac{|\lambda_{2}|(1-(-1)^{M}|\lambda_{2}|^{M})}{1+|\lambda_{2}|}\}\\
		=&\max\{\frac{1}{1-\lambda_{2}},\frac{|\lambda_{2}|(1+|\lambda_{2}|^{M})}{1+|\lambda_{2}|}\}\leq\max\{1,\frac{1}{1-\lambda_{2}}\},
	\end{align*}
	and
	\begin{align*}
		|\mathcal{C}(M,\lambda_{\Lambda})|\leq&\max\{1,| \frac{\lambda_{\Lambda}-\lambda_{\Lambda}^{M+1}}{1-\lambda_{\Lambda}}|\}=\max\{1,\frac{|\lambda_{\Lambda}||1-(-1)^{M}|\lambda_{\Lambda}|^{M}|}{1+|\lambda_{\Lambda}|}\}\\
		=&\max\{1,\frac{|\lambda_{\Lambda}|(1+|\lambda_{\Lambda}|^{M})}{1+|\lambda_{\Lambda}|}\}\leq1.
	\end{align*}
	Therefore, we know that for any odd integer $M$ and $2\leq i\leq \Lambda$,
	\begin{align*}
		\mathcal{C}^{2}(M,\lambda_{i})\leq\max\{1,\frac{1}{(1-\lambda_{2})^{2}}\},
	\end{align*}
	and by the Poincar\'e inequality in Corollary \ref{eq:Poincare}, we have
	\begin{align}\label{eq:oddineq}
		\left\|\frac{1}{M}\sum_{k=1}^{M} P^{k}_{\text{W}}f-E_{\mu_{N}}[f]	\right\|_{L^{2}(\mu_{N})}\leq&\frac{\|f\|_{L^{2}(\mu_N)}}{M}\max\{1,\frac{1}{1-\lambda_{2}}\}\nonumber\\
		\leq&\frac{\|f\|_{L^{2}(\mu_N)}}{M}\bigg(1+ \frac{e^{2\hat{J}}(e^{4\hat{J}}+1)}{2}\big(\frac{1}{2}+\hat{J}e^{\frac{e^{2\hat{J}}-1}{2}}\big)N\bigg).
	\end{align}
	When $M$ is an even integer, according to~\eqref{eq:pfexpression}~and~\eqref{eq:oddineq}, we obtain
	\begin{align*}
		\left\|\frac{1}{M}\sum_{k=1}^{M} P^{k}_{\text{W}}f-E_{\mu_{N}}[f]	\right\|_{L^{2}(\mu_{N})}\leq&\frac{1}{M}	\left\|\sum_{k=1}^{M-1} P^{k}_{\text{W}}f-E_{\mu_{N}}[f]\right\|_{L^{2}(\mu_{N})}+\frac{1}{M}\Vert P^{M}_{\text{W}}f-E_{\mu_{N}}[f]\Vert_{L^{2}(\mu_{N})}\\
		\leq&\frac{\|f\|_{L^{2}(\mu_N)}}{M}\max\{1,\frac{1}{1-\lambda_{2}}\}+\frac{1}{M}\sqrt{\sum_{j=2}^{\Lambda}\lambda_{j}^{2M}\langle f,\xi_{j}\rangle_{\mu_{N}}^{2}}\\
		\leq&\frac{\|f\|_{L^{2}(\mu_N)}}{M}\bigg(2+ \frac{e^{2\hat{J}}(e^{4\hat{J}}+1)}{2}\big(\frac{1}{2}+\hat{J}e^{\frac{e^{2\hat{J}}-1}{2}}\big)N\bigg).
	\end{align*}
	Moreover,
	\begin{align*}
		E^{\mathbb{P}_{\mu_{N}}}\left[\left( \frac{1}{M}\sum_{k=1}^{M}f(Y_{k})-E_{\mu_{N}}[f]\right)^{2}\right]
		=&\frac{1}{M^{2}}\sum_{k=1}^{M}E^{\mathbb{P}_{\mu_{N}}}\left[\left(f(Y_{k})-E_{\mu_{N}}[f]\right)^{2}\right]\\
		&+\frac{2}{M^{2}}\sum_{1\leq l<k\leq M }E^{\mathbb{P}_{\mu_{N}}}[(f(Y_{k})-E_{\mu_{N}}[f])(f(Y_{l})-E_{\mu_{N}}[f])].
	\end{align*}	
	Since
	\begin{align*}
		f(x)=&E_{\mu_{N}}[f]+\sum_{j=2}^{\Lambda}\langle f,\xi_{j}\rangle_{\mu_{N}}\xi_{j}(x),\qquad x\in\Omega,
	\end{align*}
	we have for $1\leq l\leq k\leq M$,	
	\begin{align*}
		E^{\mathbb{P}_{\mu_{N}}}[(f(Y_{k})-E_{\mu_{N}}[f])(f(Y_{l})-E_{\mu_{N}}[f])]=&E^{\mathbb{P}_{\mu_{N}}}\left[\left(\sum_{j=2}^{\Lambda}\langle f,\xi_{j}\rangle_{\mu_{N}}\xi_{j}(Y_{k})\right)\left(\sum_{s=2}^{\Lambda}\langle f,\xi_{s}\rangle_{\mu_{N}}\xi_{s}(Y_{l})\right)\right]\\
		=&\sum_{j=2}^{\Lambda}\langle f,\xi_{j}\rangle_{\mu_{N}}\sum_{s=2}^{\Lambda}\langle f,\xi_{s}\rangle_{\mu_{N}}E^{\mathbb{P}_{\mu_{N}}}[\xi_{j}(Y_{k})\xi_{s}(Y_{l})]\\
		=&\sum_{j=2}^{\Lambda}\langle f,\xi_{j}\rangle_{\mu_{N}}\sum_{s=2}^{\Lambda}\langle f,\xi_{s}\rangle_{\mu_{N}}\langle P_{\text{W}}^{k-l}\xi_{j},\xi_{s}\rangle_{\mu_{N}}.
	\end{align*}	
	Recall that $\{\xi_j\}_{j=1}^{\Lambda}$ is an orthonormal basis of $L^2(\mu_N)$ and that $P_{\text{W}}\xi_j=\lambda_j\xi_j$. Then
	\begin{align*}
		\langle P_{\text{W}}^{k-l}\xi_{j},\xi_{s}\rangle_{\mu_{N}}=&\lambda_j^{k-l}\langle \xi_j,\xi_s\rangle_{\mu_N}=\lambda^{k-l}_{j}\delta_{js}.
	\end{align*}	
	Hence,
	\begin{align*}
		E^{\mathbb{P}_{\mu_{N}}}[(f(Y_{k})-E_{\mu_{N}}[f])(f(Y_{l})-E_{\mu_{N}}[f])]
		=&\sum_{j=2}^{\Lambda}\lambda^{k-l}_{j}\langle f,\xi_{j}\rangle^{2}_{\mu_{N}},
	\end{align*}	
	and
	\begin{align*}
		E^{\mathbb{P}_{\mu_{N}}}\left[\left(\frac{1}{M}\sum_{k=1}^{M}f(Y_{k})-E_{\mu_{N}}[f]\right)^{2}\right]=&\frac{1}{M}\sum_{j=2}^{\Lambda}\langle f,\xi_{j}\rangle_{\mu_{N}}^{2}+\frac{2}{M^{2}}\sum_{1\leq l<k\leq M }\sum_{j=2}^{\Lambda}\lambda^{k-l}_{j}\langle f,\xi_{j}\rangle^{2}_{\mu_{N}}\\
		=&\frac{1}{M}\sum_{j=2}^{\Lambda}\langle f,\xi_{j}\rangle_{\mu_{N}}^{2}+\frac{2}{M^{2}}\sum_{j=2}^{\Lambda}\langle f,\xi_{j}\rangle^{2}_{\mu_{N}}\sum_{l=1}^{M}\sum_{k=l+1 }^{M}\lambda^{k-l}_{j}\\
		=&\frac{1}{M}\sum_{j=2}^{\Lambda}\langle f,\xi_{j}\rangle_{\mu_{N}}^{2}\left(1+\frac{2}{M}\frac{\lambda_{j}(M-1)-\lambda_{j}^{2}M+\lambda_{j}^{M+1}}{(1-\lambda_{j})^{2}}\right)\\
		=&\frac{1}{M}\sum_{j=2}^{\Lambda}\hat{C}(M,\lambda_{j})\langle f,\xi_{j}\rangle_{\mu_{N}}^{2},
	\end{align*}	
	where
	\begin{align*}
		\hat{C}(M,x):=&1+\frac{2}{M}\frac{x(M-1)-x^{2}M+x^{M+1}}{(1-x)^{2}},\qquad x\in [-1,1).
	\end{align*}	
	According to Lemma 3.2 in \cite{Rudolf2009document}, we know that the function $\hat{C}(M,\cdot)$ is  monotone increasing on $[-1,1)$. Hence, for $2\leq j\leq\Lambda$,  
	\begin{align*}
		\hat{C}(M,\lambda_{j})
		\leq&	\hat{C}(M,\lambda_{2})\leq\frac{2}{1-\lambda_{2}},
	\end{align*}	
	which implies
	\begin{align*}
		E^{\mathbb{P}_{\mu_{N}}}\left[\left(\frac{1}{M}\sum_{k=1}^{M}f(Y_{k})-E_{\mu_{N}}[f]\right)^{2}\right]\leq&\frac{1}{M}\hat{C}(M,\lambda_{2})\|f\|^{2}_{L^{2}(\mu_N)}\\
		\leq&\frac{\|f\|^{2}_{L^{2}(\mu_N)}}{M}e^{2\hat{J}}(e^{4\hat{J}}+1)\big(\frac{1}{2}+\hat{J}e^{\frac{e^{2\hat{J}}-1}{2}}\big)N.
	\end{align*}	
\end{proof}
\section{Application to the condensation of eigen microstate in the 1D Ising model}\label{sec:converspectralradius}
We apply the results obtained in Section~\ref{sec:3} to analyze the spectral behavior of the sample covariance matrix $\mathbb{K}$ constructed from microstates generated by the Wolff dynamics \cite{Wolff1989}. In particular, the functional inequalities established therein, including the log-Sobolev and Poincaré inequalities, provide the key tools for deriving the asymptotic spectral behavior of $\mathbb{K}$.
\subsection{Background and motivation}\label{sec:background}
We briefly recall the framework and simulation results of Chen Xiaosong and collaborators \cite{hu2019ducument,sun2021ducument}. In \cite{sun2021ducument}, Chen et al. developed an approach based on the sample covariance matrix and its associated eigen microstates to analyze the emerging phenomena and dynamic evolution of the complex systems without considering the Hamiltonian. To demonstrate the effectiveness of this approach, they performed Monte Carlo simulations of the 1D, 2D, and 3D Ising models in \cite{hu2019ducument,sun2021ducument}, where the microstates are generated using the Wolff dynamics \cite{Wolff1989}.

Their simulations indicate that in an ensemble without localization of microstate, all eigenvalues of the sample covariance matrix vanish as $M\rightarrow+\infty$ and $N\rightarrow+\infty$. In contrast, if the largest eigenvalue of the sample covariance matrix becomes dominant and converges to a finite nonzero limit as $M\rightarrow+\infty$ and $N\rightarrow+\infty$, there is a condensation of eigen microstate which is similar to the Bose-Einstein condensation, this condensation seems to show the appearance of a new phase and is considered as a phase transition, see e.g. \cite{hu2019ducument,sun2021ducument}. Furthermore, this framework has been applied to a variety of complex systems, including the price fluctuation in stock markets, the variation of ozone, the quantum Rabi model, and self-organized criticality; see e.g. \cite{Chen2021document,hu2019ducument,hu2023ducument,sun2021ducument,zhang2024document,Liu2022document} for related discussions.

These observations raise a natural theoretical question: can the emergence of a dominant eigenvalue be derived rigorously from quantitative properties of the Wolff dynamics? 
Motivated by this question, we provide a rigorous analysis of the spectral properties of the sample covariance matrix generated by the Wolff dynamics for the 1D Ising model.
\subsection{The spectrum of the sample covariance matrix at the fixed points of the RG transformation}\label{sec:renor}
In this section, we analyze the sample covariance matrix $\mathbb{K}$ at the fixed points $\hat{J}=0$ and $\hat{J}=+\infty$ of the RG transformation in the 1D Ising model. In particular, at the critical point $\hat{J}=+\infty$, we show that the largest eigenvalue $\lambda_1(\mathbb{K})$ of the sample covariance matrix $\mathbb{K}$  satisfies
\[
\lim_{M \to \infty} \lambda_1(\mathbb{K}) = 1,
\]
almost surely for any fixed positive integer $N$. This establishes the existence of a finite nonzero limit and is consistent with the simulation results.
\begin{theorem}\label{thm:fixedpoint}
	Let $(\{Y_{k}\}_{k\in\mathbb{Z}^{+}},\mathbb{P})$ denote the Markov chain considered in Theorem \ref{thm: transitionmatrix}, and let $\mu_{N}$ be the Gibbs measure of the 1D Ising model defined in Equation~\eqref{eq:isinggibbs}. For any $x\in\mathbb{R}^{N}$, define the function $g_{x}$ on $\Omega$ by
	\begin{align*}
		g_{x}(\sigma)=\langle \sigma	,x\rangle^{2}_{\mathbb{R}^{N}}, \qquad\sigma\in\Omega,
	\end{align*}
	where $\langle \cdot,\cdot\rangle_{\mathbb{R}^{N}}$ denotes the standard inner product on  $\mathbb{R}^{N}$. Then for the coupling constant $\hat{J}=0$ and any $x\in\mathbb{R}^{N}$ with $\Vert x\Vert_{2}^{2}=\langle x,x\rangle_{\mathbb{R}^{N}}\neq 0$, we have
	\begin{align}\label{eq:3.4}
		E_{\mu_{N}}[g_{x}]&=\Vert x\Vert_{2}^{2},
	\end{align}
	and the largest eigenvalue $\lambda_{1}(\mathbb{K})$ of the sample covariance matrix $\mathbb{K}$ satisfies
	\begin{align*}	
		\lim_{M\rightarrow+\infty}\lambda_{1}(\mathbb{K})=\frac{1}{N},\qquad \mathbb{P}\text{-a.s.}.
	\end{align*}
	Furthermore, for the coupling constant $\hat{J}=+\infty$ (critical point), we obtain that for any positive integer $N$, 
	\begin{align*}
		\lim\limits_{M\rightarrow+\infty}\lambda_{1}(\mathbb{K})=1,\qquad \mathbb{P}\text{-a.s.},
	\end{align*}
	and the second eigenvalue $\lambda_{2}(\mathbb{K})$ of the sample covariance matrix $\mathbb{K}$ satisfies
	\begin{align*}
		\lim\limits_{M\rightarrow+\infty}\lambda_{2}(\mathbb{K})=0, \qquad\mathbb{P}\text{-a.s.}.
	\end{align*}
\end{theorem}
\begin{proof}
	We first consider the case $x=\frac{1}{\sqrt{N}}(+\textbf{1})$. Then
	\begin{align*}
		E_{\mu_{N}}[g_{x}]=\frac{1}{N}\sum_{i=0}^{N}(N-2i)^{2}\mu_{N}(\#\{\mathcal{C}^{\sigma}\}=i),
	\end{align*}
	where $\mathcal{C}^{\sigma}$ is the set defined in~\eqref{eq:decomposet}~and $\#\{\mathcal{C}^{\sigma}\}$ is the number of $+1$ spins in $\sigma$. For $\hat{J}=0$,
	\begin{align*}
		\mu_{N}(\#\{\mathcal{C}^{\sigma}\}=i)=\frac{1}{2^{N}}C_{N}^{i}=\frac{1}{\Lambda}C_{N}^{i},
	\end{align*}
	and thus
	\begin{align*}
		E_{\mu_{N}}[g_{x}]=&\frac{1}{N}\sum_{i=0}^{N}(N-2i)^{2}\mu_{N}(\#\{\mathcal{C}^{\sigma}\}=i)\\
		=&\frac{1}{N}\sum_{i=0}^{N}\frac{C_{N}^{i}}{\Lambda}(N-2i)^{2}
		=N+\frac{4}{N}\sum_{i=0}^{N}\frac{C_{N}^{i}}{\Lambda}i^{2}-4\sum_{i=0}^{N}\frac{C_{N}^{i}}{\Lambda}i=1. 	
	\end{align*}
Moreover, for convenience, we list the configurations in $\Omega$ in binary order
\begin{align}\label{eq:2.7}
	&S^{T}=\left(                 
	\begin{array}{ccccccc}
		\\
		(S^{1})^{T}\\
		(S^{2})^{T}\\	
		(S^{3})^{T}\\
		(S^{4})^{T}\\
		\vdots \\
		(S^{2^{N-1}})^{T}\\
		\vdots\\
		(S^{2^{N}})^{T}
	\end{array}
	\right)=            
	\left(                 
	\begin{array}{rrrrrrrrrrr}
		\sigma_{1}&\sigma_{2}&\cdots&\cdots&\sigma_{N-1}&\sigma_{N}\\
		1 &1 &\cdots&\cdots&1 &1\\
		-1&1 &\cdots&\cdots&1 &1\\	
		1 &-1 &\cdots&\cdots&1 &1\\
		-1&-1 &\cdots&\cdots&1 &1\\
		\qquad &\vdots &\qquad &\vdots\\
		1 &1 &\cdots&\cdots&1 &-1\\
		\qquad &\vdots &\qquad &\vdots\\
		-1&-1 &\cdots&\cdots&-1 &-1
	\end{array}
	\right).            
\end{align}	
Then, for the case $\hat{J}=0$ and any $x\in \mathbb{R^{N}}$ with $\Vert x\Vert_{2}\neq 0$,
\begin{align*}
	\frac{1}{\vert|x\vert|_{2}^{2}}E_{\mu_{N}}[g_{x}]=&\frac{1}{\vert|x\vert|_{2}^{2}}\sum_{j=1}^{\Lambda}\left(\sum_{i=1}^{N}x_{i}S^{j}_{i}\right)^{2}\mu_{N}(S^{j})=\frac{1}{\vert|x\vert|_{2}^{2}}\sum_{j=1}^{\Lambda}\left(\sum_{i=1}^{N}x_{i}S^{j}_{i}\right)^{2}\frac{1}{\Lambda}\\
	=&\frac{1}{\Lambda\vert|x\vert|_{2}^{2}}\sum_{j=1}^{\Lambda}\left(\sum_{k=1}^{N}x_{k}S^{j}_{k}\right)\left(\sum_{i=1}^{N}x_{i}S^{j}_{i}\right)=\frac{1}{\Lambda\vert|x\vert|_{2}^{2}}\sum_{k=1}^{N}\sum_{i=1}^{N}x_{i}x_{k}\left(\sum_{j=1}^{\Lambda}S^{j}_{k}S^{j}_{i}\right),
\end{align*}
where $S_{i}^{j}$ is the $i$-th component of $S^{j}$. If $i\neq k$,
\begin{align*}	
	\sum_{j=1}^{\Lambda}S^{j}_{k}S^{j}_{i}=\sum_{\sigma_{1}=\pm1}\cdots\sum_{\sigma_{N}=\pm1}\sigma_{k}\sigma_{i}=\sum_{\sigma_{j}=\pm1,j\neq i,k}\left(\sum_{\sigma_{i}=\pm1}\sum_{\sigma_{k}=\pm1}\sigma_{k}\sigma_{i}\right)=0.
\end{align*}
If $i=k$, 
\begin{align*}	
	\sum_{j=1}^{\Lambda}S^{j}_{i}S^{j}_{i}=\sum_{\sigma_{1}=\pm1}\cdots\sum_{\sigma_{N}=\pm1}\sigma^{2}_{i}=2^{N}=\Lambda.
\end{align*}
Hence, we obtain Equation~\eqref{eq:3.4}
\begin{align*}
	\frac{1}{\vert|x\vert|_{2}^{2}}E_{\mu_{N}}[g_{x}]=&\frac{1}{\Lambda\vert|x\vert|_{2}^{2}}\sum_{k=1}^{N}\sum_{i=1}^{N}x_{i}x_{k}\left(\sum_{j=1}^{\Lambda}S^{j}_{k}S^{j}_{i}\right)=\frac{1}{\vert|x\vert|_{2}^{2}}\sum_{i=1}^{N}x_{i}^{2}=1.
\end{align*}
Denote by $\eta_{1}$ the unit eigenvector corresponding to the largest eigenvalue $\lambda_{1}(\mathbb{K})$ of the sample covariance matrix $\mathbb{K}$. Therefore,
\begin{align*}
	\lambda_{1}(\mathbb{K})=\eta_{1}^{T}\mathbb{K}\eta_{1}	=\frac{1}{M}\sum_{k=1}^{M}\eta_{1}^{T}X_{k} X^{T}_{k}\eta_{1}=\frac{1}{M}\sum_{k=1}^{M}\langle X_{k},\eta_{1}\rangle^{2}_{\mathbb{R}^{N}}.
\end{align*}
By Theorem \ref{thm:birkhoff} and Equation~\eqref{eq:3.4}, we have
\begin{align*}
	\lambda_{1}(\mathbb{K})=&\eta_{1}^{T}\mathbb{K}\eta_{1}=\frac{1}{M}\sum_{k=1}^{M}\langle X_{k},\eta_{1}\rangle^{2}_{\mathbb{R}^{N}}\\
	=&\frac{1}{MN}\sum_{k=1}^{M}g_{\eta_{1}}(Y_{k})\rightarrow \frac{1}{N} E_{\mu_{N}}[g_{\eta_{1}}]=\frac{1}{N},\qquad\mathbb{P}\text{-a.s.},
\end{align*}
as $M\rightarrow +\infty$.

If the coupling constant $\hat{J}=+\infty$ (critical point), by Theorem \ref{thm:infiniteJ}, for any initial configuration $Y_{1}=\tilde{\sigma}\in\Omega$ and $k\geq 1$,
\begin{align*}
	\mathbb{P}(Y_{\tilde{\mathcal{C}}^{\tilde{\sigma}}+k}=+\textbf{1}\mid Y_{1}=\tilde{\sigma})+\mathbb{P}(Y_{\tilde{\mathcal{C}}^{\tilde{\sigma}}+k}=-\textbf{1}\mid Y_{1}=\tilde{\sigma})=1,
\end{align*}
which implies
\begin{align*}
	\mathbb{P}(Y_{\tilde{\mathcal{C}}^{\tilde{\sigma}}+k}Y_{\tilde{\mathcal{C}}^{\tilde{\sigma}}+k}^{T}=\textbf{1}\textbf{1}^{T}\mid Y_{1}=\tilde{\sigma})=1,\qquad k\geq 1.
\end{align*}
Hence, for any given initial configuration $Y_{1}=\tilde{\sigma}\in\Omega$ and $M\geq N$,
\begin{align*}
	\mathbb{K}=&\frac{1}{M}\sum_{k=1}^{M}X_{k} X^{T}_{k}=\frac{1}{MN}\sum_{k=1}^{M}Y_{k} Y^{T}_{k}=\frac{1}{MN}\sum_{k=1}^{\tilde{\mathcal{C}}^{\tilde{\sigma}}}Y_{k} Y^{T}_{k}+\frac{1}{MN}\sum_{k=\tilde{\mathcal{C}}^{\tilde{\sigma}}+1}^{M}Y_{k} Y^{T}_{k}\\
	=&\frac{1}{MN}\sum_{k=1}^{\tilde{\mathcal{C}}^{\tilde{\sigma}}}Y_{k} Y^{T}_{k}+\frac{M-\tilde{\mathcal{C}}^{\tilde{\sigma}}}{M}\frac{1}{N}\textbf{1}\textbf{1}^{T},\qquad \mathbb{P}\text{-a.s.}.
\end{align*}
Let $\lambda_{2}(A)$ denote the second largest eigenvalue of a matrix $A$. By Weyl’s inequality \cite{Weyl1912document}, we get for $M\geq N$
\begin{align*}
	\lambda_{1}(\mathbb{K})\geq&\frac{M-\tilde{\mathcal{C}}^{\tilde{\sigma}}}{M}\lambda_{1}(\frac{1}{N}\textbf{1}\textbf{1}^{T})=\frac{M-\tilde{\mathcal{C}}^{\tilde{\sigma}}}{M}\geq1-\frac{N}{M},\qquad \mathbb{P}\text{-a.s.},
\end{align*}
and 
\begin{align*}
	0\leq\lambda_{2}(\mathbb{K})\leq&\lambda_{2}(\mathbb{K}-\frac{1}{MN}\sum_{k=1}^{\tilde{\mathcal{C}}^{\tilde{\sigma}}}Y_{k} Y^{T}_{k})+\frac{1}{MN}\lambda_{1}(\sum_{k=1}^{\tilde{\mathcal{C}}^{\tilde{\sigma}}}Y_{k} Y^{T}_{k})\\
	=&\frac{M-\tilde{\mathcal{C}}^{\tilde{\sigma}}}{M}\lambda_{2}(\frac{1}{N}\textbf{1}\textbf{1}^{T})+\frac{1}{MN}\lambda_{1}(\sum_{k=1}^{\tilde{\mathcal{C}}^{\tilde{\sigma}}}Y_{k} Y^{T}_{k})\\
	\leq&\frac{M-\tilde{\mathcal{C}}^{\tilde{\sigma}}}{M}\lambda_{2}(\frac{1}{N}\textbf{1}\textbf{1}^{T})+\frac{1}{MN}\text{Tr}(\sum_{k=1}^{\tilde{\mathcal{C}}^{\tilde{\sigma}}}Y_{k} Y^{T}_{k})
	=\frac{\tilde{\mathcal{C}}^{\tilde{\sigma}}}{M}\leq\frac{N}{M},\qquad \mathbb{P}\text{-a.s.}.
\end{align*}
Hence, 
\begin{align*}
	\lim\limits_{M\rightarrow+\infty}\lambda_{1}(\mathbb{K})=1, \qquad	\lim\limits_{M\rightarrow+\infty}\lambda_{2}(\mathbb{K})=0, \qquad\mathbb{P}\text{-a.s.}.
\end{align*}
\end{proof}
\begin{remark}\label{rem:critical}
The above result shows that, at the critical point $\hat{J}=+\infty$, the largest eigenvalue of the sample covariance matrix $\mathbb{K}$ converges to a finite nonzero limit as $M\rightarrow +\infty$ for any fixed $N$, which implies that $\lambda_1(\mathbb{K})$ admits a finite nonzero limit 
under the iterated limit $M\rightarrow +\infty$ and $N\rightarrow +\infty$. This behavior is in sharp contrast to that in the subcritical regime, where the spectral radius of $\mathbb{K}$ will be shown to vanish.
\end{remark}
\subsection{Spectral estimation of the sample covariance matrix in the entire subcritical regime under the iterated limit}\label{sec:sec:almost}
We now analyze the spectral behavior of the sample covariance matrix $\mathbb{K}$ in the entire subcritical regime $\hat{J}\in[0,+\infty)$, which corresponds to the temperature range  $T> T_{c}=0$ for the 1D Ising model. As suggested by the simulation results discussed in Section~\ref{sec:background}, all eigenvalues of the sample covariance matrix $\mathbb{K}$ are expected to vanish as $M \to +\infty$ and $N \to +\infty$ in this regime.  In this section, we confirm this behavior rigorously by proving that the spectral radius $\rho(\mathbb{K})$ of the sample covariance matrix $\mathbb{K}$ converges to $0$ almost surely under the iterated limit $M \to +\infty$, $N \to +\infty$. This provides theoretical support for Chen's simulation results.

Recall from \eqref{eq:defcov} that, for $1 \leq i,j \leq N$,
\begin{align*}	\mathbb{K}_{ij}=\frac{1}{M}\sum_{k=1}^{M}X_{ik}	X_{jk}=\frac{1}{MN}\sum_{k=1}^{M}Y_{ik}	Y_{jk}.
\end{align*}
To control the spectral radius of $\mathbb{K}$, we introduce the matrix norm
\[
\|\mathbb{K}\|_{1} := \max_{1 \leq i \leq N} \sum_{j=1}^{N} |\mathbb{K}_{ij}|.
\]
By the Gershgorin circle theorem, the spectral radius $\rho(\mathbb{K})$ satisfies
\[
0 \leq \rho(\mathbb{K}) \leq \|\mathbb{K}\|_{1}.
\]
\begin{theorem}\label{thm:almost}
Consider the 1D Ising model with coupling constant $\hat{J}\in[0,+\infty)$. Let $(Y_{k})_{k\in\mathbb{Z}^{+}}$ and $\mu_{N}$ be defined as in Theorem \ref{thm: transitionmatrix}, the sample covariance matrix $\mathbb{K}$ satisfies
\begin{align*}	\mathbb{P}(\lim\limits_{N\rightarrow +\infty}\lim\limits_{M\rightarrow +\infty}	\rho(\mathbb{K})=0)=\mathbb{P}(\lim\limits_{N\rightarrow +\infty}\lim\limits_{M\rightarrow +\infty}\Vert\mathbb{K}\Vert_{1}=0)=1. 	
\end{align*}
Furthermore, we know that for $\hat{J}\in[0,+\infty)$
\begin{align*}
	\lim\limits_{N\rightarrow +\infty}\lim\limits_{M\rightarrow +\infty}	\lambda_{1}(\mathbb{K})=0,\qquad  \mathbb{P}\text{-a.s.}.
\end{align*} 
\end{theorem}
\begin{proof}
For $1\leq i,j\leq N$,
\begin{align*}	\mathbb{K}_{ij}=\left(\frac{1}{M}\sum_{k=1}^{M}X_{k} X^{T}_{k}\right)_{ij}=\frac{1}{MN}\sum_{k=1}^{M}Y_{ik}Y_{jk}.
\end{align*}
Let $f_{ij}$ be the real-valued function on $\Omega$ defined by $f_{ij}(\sigma)=\sigma_{i}\sigma_{j}$. Define
\begin{align*}
	\hat{\mathbb{K}}_{ij}:=\frac{1}{N}E_{\mu_{N}}[f_{ij}].
\end{align*}
By Theorem \ref{thm:birkhoff}, we have
\begin{align}\label{eq:4.6}
	\lim_{M\rightarrow +\infty}\mathbb{K}_{ij}=\lim_{M\rightarrow +\infty}\frac{1}{MN}\sum_{k=1}^{M}f_{ij}(Y_{k})=\frac{1}{N}E_{\mu_{N}}[f_{ij}]=\hat{\mathbb{K}}_{ij},\qquad \mathbb{P}\text{-a.s.}.
\end{align}
According to Equation~\eqref{eq:1DIsingcorrela}~, we know that for any $\hat{J}\in[0,+\infty)$ and $1\leq i\leq j\leq N$,
\begin{align*}	\hat{\mathbb{K}}_{ij}=\frac{1}{N}E_{\mu_{N}}[f_{ij}]=\frac{1}{N}E_{\mu_{N}}[\sigma_{i}\sigma_{j}]=\frac{1}{N}\frac{\boldsymbol{\theta}^{j-i}+\boldsymbol{\theta}^{N-j+i}}{1+\boldsymbol{\theta}^{N}}\geq 0.
\end{align*}
Then by~\eqref{eq:sumcorrelation},
\begin{align}\label{eq:upperboundsamplecov}	\sum_{j=1}^{N}|\hat{\mathbb{K}}_{ij}|=\frac{1}{N}\sum_{j=1}^{N}E_{\mu_{N}}[\sigma_{i}\sigma_{j}]&\leq  \frac{1}{N}+\frac{1}{N}\frac{2\lambda_{+}}{\lambda_{+}-\lambda_{-}}\frac{\boldsymbol{\theta}-\boldsymbol{\theta}^{N}}{1+\boldsymbol{\theta}^{N}}.
\end{align}
Applying Gershgorin circle theorem again, the spectral radius of the matrix $\hat{\mathbb{K}}$ with elements $\hat{\mathbb{K}}_{ij}$ has the following upper bound
\begin{align*}
	\rho(\hat{\mathbb{K}})\leq \Vert\hat{\mathbb{K}}\Vert_{1}= \max_{1\leq i\leq N}\{\sum_{j=1}^{N}|\hat{\mathbb{K}}_{ij}|\}\leq \frac{1}{N}+\frac{1}{N}\frac{2\lambda_{+}}{\lambda_{+}-\lambda_{-}}\frac{\boldsymbol{\theta}-\boldsymbol{\theta}^{N}}{1+\boldsymbol{\theta}^{N}}.
\end{align*}	
Taking the limit as $N\rightarrow +\infty$, we obtain
\begin{align*}	\lim\limits_{N\rightarrow +\infty}\rho(\hat{\mathbb{K}})=\lim\limits_{N\rightarrow +\infty}	\Vert\hat{\mathbb{K}}\Vert_{1}= 0.
\end{align*}
According to Equation~\eqref{eq:4.6}, we get
\begin{align*}	\lim\limits_{N\rightarrow +\infty}\lim\limits_{M\rightarrow +\infty}\Vert\mathbb{K}\Vert_{1}&\leq \lim\limits_{N\rightarrow +\infty}\lim\limits_{M\rightarrow +\infty}\max_{1\leq i\leq N}\{\sum_{j=1}^{N}|\mathbb{K}_{ij}|\}=\lim\limits_{N\rightarrow +\infty}\max_{1\leq i\leq N}\{\sum_{j=1}^{N}|\hat{\mathbb{K}}_{ij}|\}\\
	&\leq \lim\limits_{N\rightarrow +\infty}\frac{1}{N}+\lim\limits_{N\rightarrow +\infty}\frac{1}{N}\frac{2\lambda_{+}}{\left(\lambda_{+}-\lambda_{-}\right)}\frac{\boldsymbol{\theta}-\boldsymbol{\theta}^{N}}{1+\boldsymbol{\theta}^{N}}=0,\qquad \mathbb{P}\text{-a.s.}.
\end{align*}
Finally, 
\begin{align*}	\lim\limits_{N\rightarrow +\infty}\lim\limits_{M\rightarrow +\infty} \lambda_{1}(\mathbb{K})= 0,\qquad \mathbb{P}\text{-a.s.}.
\end{align*}
\end{proof}
\begin{remark}
The above result shows that, in the entire subcritical regime $\hat{J}\in[0,+\infty)$, the spectral radius of the sample covariance matrix vanishes under the iterated limit. This is in contrast to the behavior at the critical point $\hat{J}=+\infty$, as noted in Remark~\ref{rem:critical}. These two distinct asymptotic behaviors are consistent with the phenomenon observed in Chen's simulations.
\end{remark}
\begin{remark}
The convergence of the spectral radius $\rho(\mathbb{K})$ is closely related to the decay behavior of correlation functions. Such estimates have been extensively studied for a wide range of models, including unbounded spin systems (see e.g. \cite{Bodineau1999document,Helffer1998ducument,Helffer1999ducument,Yoshida1999document,Menz2014document,Otto2007document}), the high dimensional Ising models (see e.g. \cite{Bauerschmidt2024document,Ding2023document,Duminil2020document,Ott2023document}) and the lattice Yang-Mills fields (see e.g. \cite{Adhikari2025document,Shen2021document,Shen2024document}). This suggests that the vanishing of the spectral radius $\rho(\mathbb{K})$ under the iterated limit may extend to more general interacting systems.
\end{remark}
\subsection{Spectral estimation of the sample covariance matrix in the entire subcritical regime under the double limit}\label{sec:4}
Having established the almost sure convergence of the spectral radius $\rho(\mathbb{K})$ under the iterated limit in the previous section, we now investigate whether the same behavior holds under the double limit. To this end, we derive an explicit upper bound for $\Vert\mathbb{K}\Vert_{1}$ in $L^{2}(\mathbb{P}_{\mu_{N}})$, which implies that the spectral radius $\rho(\mathbb{K})$ converges to 0 in $L^{2}(\mathbb{P}_{\mu_{N}})$ under the double limit $M,N\to+\infty$ with $N=o(\sqrt{M})$. This shows that, under both the iterated limit and the double limit, the spectral radius of the sample covariance matrix converges to $0$ in the entire subcritical regime.
\begin{theorem}\label{thm:doubleinproba}Consider the 1D Ising model with coupling constant $\hat{J}\in[0,+\infty)$. Let $(Y_{k})_{k\in\mathbb{Z}^{+}}$ and $\mu_{N}$ be defined as in Theorem \ref{thm: transitionmatrix}. Then for $N=o(\sqrt{M})$ (i.e., $N/\sqrt{M} \to 0$ as $M\to+\infty$),
\begin{align*}
	\lim_{N,M\rightarrow+\infty}E^{\mathbb{P}_{\mu_{N}}}[\Vert\mathbb{K}\Vert^{2}_{1}]=0,
\end{align*}
which implies
\begin{align*}
	\lim_{N,M\rightarrow+\infty}E^{\mathbb{P}_{\mu_{N}}}[\rho(\mathbb{K})]=0.
\end{align*}
\end{theorem}
\begin{proof}
According to the estimates in Theorem \ref{thm: chernoffbound}, we have
\begin{align*}
	E^{\mathbb{P}_{\mu_{N}}}\left[\left( \frac{1}{M}\sum_{k=1}^{M}f(Y_{k})-E_{\mu_{N}}[f]\right)^{2}\right]
	\leq\frac{\|f\|^{2}_{L^{2}(\mu_N)}}{M}e^{2\hat{J}}(e^{4\hat{J}}+1)\big(\frac{1}{2}+\hat{J}e^{\frac{e^{2\hat{J}}-1}{2}}\big)N.
\end{align*}	
Taking $f$ to be the real-valued function $f_{ij}$ with $f_{ij}(\sigma)=\sigma_i\sigma_j$, we have
\begin{align*}	
	E^{\mathbb{P}_{\mu_{N}}}[|\mathbb{K}_{ij}-\hat{\mathbb{K}}_{ij}|^{2}]
	=&E^{\mathbb{P}_{\mu_{N}}}[|\frac{1}{MN}\sum_{k=1}^{M}Y_{ik}Y_{jk}-\frac{1}{N}E_{\mu_{N}}[f_{ij}]|^{2}]\\
	\leq& \frac{1}{MN}e^{2\hat{J}}(e^{4\hat{J}}+1)\big(\frac{1}{2}+\hat{J}e^{\frac{e^{2\hat{J}}-1}{2}}\big).
\end{align*}
Therefore, for any $1\leq i\leq N$,
\begin{align*}
	\sqrt{E^{\mathbb{P}_{\mu_{N}}}\left[\left(\sum_{j=1}^{N}|\mathbb{K}_{ij}|-\sum_{j=1}^{N}|\hat{\mathbb{K}}_{ij}|\right)^{2}\right]}
	\leq&\sum_{j=1}^{N}\sqrt{	E^{\mathbb{P}_{\mu_{N}}}[( |\mathbb{K}_{ij}|-|\hat{\mathbb{K}}_{ij}|)^{2}]}
	\leq\sum_{j=1}^{N}\sqrt{	E^{\mathbb{P}_{\mu_{N}}}[| \mathbb{K}_{ij}-\hat{\mathbb{K}}_{ij}|^{2}]}\\
	\leq&\sqrt{\frac{N}{M}e^{2\hat{J}}(e^{4\hat{J}}+1)\big(\frac{1}{2}+\hat{J}e^{\frac{e^{2\hat{J}}-1}{2}}\big)}.
\end{align*}
By~\eqref{eq:upperboundsamplecov}, we have
\begin{align*}	\sum_{j=1}^{N}|\hat{\mathbb{K}}_{ij}|&\leq \frac{1}{N}+\frac{1}{N}\frac{2\lambda_{+}}{\left(\lambda_{+}-\lambda_{-}\right)}\frac{\boldsymbol{\theta}-\boldsymbol{\theta}^{N}}{1+\boldsymbol{\theta}^{N}}.
\end{align*}
Hence,
\begin{align*}
	\sqrt{E^{\mathbb{P}_{\mu_{N}}}[|\sum_{j=1}^{N}|\mathbb{K}_{ij}||^{2}]}
	\leq&\sqrt{E^{\mathbb{P}_{\mu_{N}}}\left[\left(\sum_{j=1}^{N}|\mathbb{K}_{ij}|-\sum_{j=1}^{N}|\hat{\mathbb{K}}_{ij}|\right)^{2}\right]}+\sqrt{E^{\mathbb{P}_{\mu_{N}}}[|\sum_{j=1}^{N}|\hat{\mathbb{K}}_{ij}||^{2}]}\\
	\leq&\sqrt{\frac{N}{M}e^{2\hat{J}}(e^{4\hat{J}}+1)\big(\frac{1}{2}+\hat{J}e^{\frac{e^{2\hat{J}}-1}{2}}\big)}+\frac{1}{N}+\frac{1}{N}\frac{2\lambda_{+}}{\left(\lambda_{+}-\lambda_{-}\right)}\frac{\boldsymbol{\theta}-\boldsymbol{\theta}^{N}}{1+\boldsymbol{\theta}^{N}}.
\end{align*}
Applying the elementary inequality 
\begin{align*}
	\max_{1\leq i\leq N}\{a_{i}\}\leq\sqrt{\sum_{i=1}^{N}a_{i}^{2}},\qquad a_{i}\geq 0, 
\end{align*}
we obtain
\begin{align*}
	E^{\mathbb{P}_{\mu_{N}}}[\Vert\mathbb{K}\Vert^{2}_{1}]=&E^{\mathbb{P}_{\mu_{N}}}[|\max_{1\leq i\leq N }\sum_{j=1}^{N}|\mathbb{K}_{ij}||^{2}]
	\leq\sum_{i=1}^{N}E^{\mathbb{P}_{\mu_{N}}}[|\sum_{j=1}^{N}|\mathbb{K}_{ij}||^{2}]\\
	\leq&\big(\sqrt{\frac{N^{2}}{M}e^{2\hat{J}}(e^{4\hat{J}}+1)\big(\frac{1}{2}+\hat{J}e^{\frac{e^{2\hat{J}}-1}{2}}\big)}+\frac{1}{\sqrt{N}}+\frac{1}{\sqrt{N}}\frac{2\lambda_{+}}{\left(\lambda_{+}-\lambda_{-}\right)}\frac{\boldsymbol{\theta}-\boldsymbol{\theta}^{N}}{1+\boldsymbol{\theta}^{N}}\big)^{2}.
\end{align*}
Finally, for $N=o(\sqrt{M})$, we have
\begin{align*}
	\lim_{N,M\rightarrow+\infty}E^{\mathbb{P}_{\mu_{N}}}[\Vert\mathbb{K}\Vert^{2}_{1}]=0.
\end{align*}
This completes the proof of our main theorem.
\end{proof}
\begin{remark}
The spectral radius of the sample covariance matrix $\mathbb{K}$ converges to $0$ in $L^{2}(\mathbb{P}_{\mu_{N}})$ under the double limit $M,N\to+\infty$ with $N=o(\sqrt{M})$ for $\hat{J}\in[0,+\infty)$, which is consistent with the absence of a phase transition in this regime. Together with Theorem~\ref{thm:almost} and Theorem~\ref{thm:doubleinproba}, this shows that, under both the iterated limit and the double limit, the spectral radius of $\mathbb{K}$ vanishes in the entire subcritical regime. This provides further theoretical support for Chen's simulation results.
\end{remark}

\subsection{Theoretical results and consistency with simulations}\label{sec:interpretation}
Chen's simulations indicate a sharp contrast in the spectral behavior of the sample covariance matrix $\mathbb{K}$ between the subcritical regime and the critical point. Taking the 1D Ising model as an example, in the subcritical regime $\hat{J}\in[0,+\infty)$, the largest eigenvalue $\lambda_{1}(\mathbb{K})$ of the sample covariance matrix $\mathbb{K}$ vanishes as $M\to+\infty$ and $N\to+\infty$, while at the critical point $\hat{J}=+\infty$, it converges to a finite nonzero limit. This indicates a condensation of eigen microstate and is considered as a phase transition, with the new phase characterized by the eigen microstate corresponding to $\lambda_{1}(\mathbb{K})$, as noted by Chen et al. in recent studies \cite{hu2019ducument,sun2021ducument}.

Our results are consistent with these observations and provide a rigorous characterization of this behavior in the 1D Ising model. At the critical point $\hat{J}=+\infty$, the largest eigenvalue $\lambda_{1}(\mathbb{K})$ converges to 1 almost surely as $M\rightarrow+\infty$ for any fixed $N$, and thus admits a finite nonzero limit under $M\rightarrow +\infty$, $N\rightarrow +\infty$. In contrast, for $\hat{J}\in[0,+\infty)$, the spectral radius of $\mathbb{K}$ vanishes almost surely under the iterated limit $M\to+\infty$, $N\to+\infty$, and also vanishes in $L^{2}(\mathbb{P}_{\mu_N})$ under the double limit $M,N\to+\infty$ with $N=o(\sqrt{M})$. These results show a clear distinction between the critical and subcritical regimes, which exhibit fundamentally different spectral behaviors of the sample covariance matrix $\mathbb{K}$. The two asymptotic behaviors of  $\mathbb{K}$ agree with those reported in Chen's simulations for the 1D Ising model, thus providing theoretical support for their simulation results.

\section{Acknowledgement}
This work was supported by National Natural Science Foundation of China (Grant No.12288201).

\end{document}